\documentclass[12pt, reqno,draft]{amsart}
\usepackage{amsmath}
\usepackage{amssymb}
\usepackage{amsthm}
\usepackage{enumerate}
\usepackage{xcolor}
\usepackage{mathrsfs}
\usepackage[abbrev]{amsrefs}
\usepackage[utf8]{inputenc}
\usepackage{ifthen}
\usepackage{enumitem}
\usepackage{bm}
\usepackage{graphicx}
\usepackage[all]{xy}

\newtheorem{thm}{}[section]
\newtheorem{theorem}[thm]{Theorem}
\newtheorem{corollary}[thm]{Corollary}
\newtheorem{lemma}[thm]{Lemma}
\newtheorem{proposition}[thm]{Proposition}

\theoremstyle{definition}
\newtheorem{definition}[thm]{Definition}
\theoremstyle{remark}
\newtheorem{remark}[thm]{Remark}

\newtheorem{question}[thm]{Question}

\numberwithin{equation}{section}
\allowdisplaybreaks

\newcommand{\II}{\ensuremath{\mathcal{I}}}

\newcommand{\Nt}{\ensuremath{\mathcal{N}}}
\newcommand{\Pt}{\ensuremath{\mathcal{P}}}
\newcommand{\Mt}{\ensuremath{\mathcal{M}}}
\newcommand{\Ts}{\ensuremath{\mathcal{T}}}
\newcommand{\UU}{\ensuremath{\mathbb{U}}}

\newcommand{\Sp}{\ensuremath{\mathcal{S}}}
\newcommand{\EE}{\ensuremath{\mathbb{E}}}
\newcommand{\RR}{\ensuremath{\mathbb{R}}}
\newcommand{\NN}{\ensuremath{\mathbb{N}}}
\newcommand{\FF}{\ensuremath{\mathbb{F}}}
\newcommand{\LL}{\ensuremath{\bm{L}}}
\newcommand{\RRt}{\ensuremath{\mathcal{R}}}
\newcommand{\Cn}{\ensuremath{\mathcal{C}}}
\newcommand{\GG}{\ensuremath{\mathscr{G}}}
\newcommand{\ee}{\ensuremath{\bm{e}}}
\newcommand{\Id}{\ensuremath{\mathrm{Id}}}

\newcommand{\xx}{\ensuremath{\bm{x}}}
\newcommand{\XX}{\ensuremath{\mathbb{X}}}
\newcommand{\YY}{\ensuremath{\mathbb{Y}}}
\newcommand{\LT}{\ensuremath{\mathcal{L}}}
\newcommand{\Cu}{\ensuremath{\mathcal{Q}}}

\newcommand{\Ft}{\ensuremath{\mathcal{F}}}
\newcommand{\VV}{\ensuremath{\mathbb{V}}}
\newcommand{\yy}{\ensuremath{\bm{y}}}

\AtBeginDocument{
\def\MR#1{}
}

\begin{document}

\title[Integrals of quasi-Banach-space-valued functions]{Toward an optimal theory of integration for quasi-Banach-space-valued functions}

\author[J.~L. Ansorena]{Jos\'e L. Ansorena}
\address{Department of Mathematics and Computer Sciences\\
Universidad de La Rioja\\
Logro\~no\\
26004 Spain}
\email{joseluis.ansorena@unirioja.es}

\author[G. Bello]{Glenier Bello}
\address{Institute of Mathematics of the Polish Academy of Sciences\\
00-656 Warszawa\\
ul. \'{S}niadeckich 8\\
Poland}
\email{gbello@impan.pl}



\begin{abstract}
We present a new approach to define a suitable integral for functions with 
values in quasi-Banach spaces. The integrals of Bochner and Riemann have 
deficiencies in the non-locally convex setting. 
The study of an integral for $p$-Banach spaces initiated by Vogt is neither totally satisfactory, 
since there are quasi-Banach spaces which are $p$-convex for all $0<p<1$, 
so it is not always possible to choose an optimal $p$ to develop the integration. 
Our method puts the emphasis on the galb of the space, which permits 
a precise definition of its convexity. 
The integration works for all spaces of galbs known in the literature. 
We finish with a fundamental theorem of calculus for our integral. 

\end{abstract}

\thanks{J.~L. Ansorena acknowledges the support of the Spanish Ministry for Science, Innovation, and Universities under Grant PGC2018-095366-B-I00 for \emph{An\'alisis Vectorial, Multilineal y Aproximaci\'on}. G. Bello was supported by National Science Centre, Poland grant UMO-2016/21/B/ST1/00241.}

\maketitle

\section{Introduction}
\noindent
If $\XX$ is a non-locally convex space, it is easy to 
construct a sequence of simple functions 
\[
s_n\colon[0,1]\to\XX, \quad s_n(t)=\sum_{m=1}^{n}\chi_{A_{m,n}}(t)x_{m,n}, 
\]
where $(A_{m,n})_{m=1}^{n}$ is a partition of the interval $[0,1]$ for each $n\in\NN$, 
and $\chi$ denotes the characteristic function, such that 
\[
\sup_{1\le m\le n}\lVert x_{m,n}\rVert \to 0, \quad 
\sum_{m=1}^{n}\mu(A_{m,n})x_{m,n}\nrightarrow  0,
\] 
as $n$ goes to infinity, where $\mu$ denotes the Lebesgue measure
(cf.\ \cite{RolewiczBOOK1985}*{pp.\ 121-123}). 
Therefore, Bochner-Lebesgue integration 
cannot be extended to non-locally convex spaces. 
On the other hand, the definition of the Riemann integral extends verbatim 
for functions defined on an interval $[a,b]$ with values in an $F$-space $\XX$. 
However, it has some problems in the non-locally convex setting. 
For example, Mazur and Orlicz \cite{MO1942} proved that the $F$-space $\XX$ 
is non-locally convex if and only if there is a continuous function 
$f\colon[0,1]\to\XX$ which is not Riemann integrable. 
But the main drawback is that 
the Riemann integral operator $\mathcal{I}_\mathcal{R}$, 
acting from the set of $\XX$-valued simple functions $\mathcal{S}([a,b],\XX)$ 
to $\XX$ by 
\[ 
\mathcal{I}_\mathcal{R}\Big(\sum_{j=1}^{n} x_j \chi_{[t_{j-1},t_j)}\Big)
=
\sum_{j=1}^{n}(t_j-t_{j-1})x_j, 
\]
is not continuous when $\XX$ is not locally convex (see \cite{AA2013}*{Theorem~2.3}). 

An important attempt (somehow missed in the literature) to develop 
a theory of integration based on operators for functions with values 
in a quasi-Banach-space (i.e.\ a locally bounded $F$-space)
was initiated by Vogt \cite{Vogt1967}. 
A remarkable theorem of Aoki and Rolewicz \cites{Aoki1942, Rolewicz1957} 
says that any quasi-normed space is $p$-convex for some $0<p\le 1$. 
The idea of Vogt was the following. 
Given a quasi-Banach space $\XX$, let $0<p\le 1$ be such that $\XX$ is $p$-convex. 
For this fixed $p$, he developed a theory of integration 
based on an identification of tensor spaces with function spaces 
(see \cite{Vogt1967}*{Satz~4}). 
Among the papers that approach integration of quasi-Banach-valued functions from Vogt's point of view we highlight  \cite{Maurey1972}.

The main advantage of Vogt's integration with respect other approaches to integration in the non locally convex setting is that it provides a bounded operator from the space of integrable functions into the target quasi-Banach space. Regarding the limitations, its main drawback is that it depends heavily on the convexity parameter $p$ chosen, and for some spaces there is no optimal choice of $p$. Take, for instance, the weak Lorentz space $L_{1,\infty}=L_{1,\infty}(\RR)$. This classical space, despite not being locally convex, is $p$-convex for any $0<p<1$ (see \cite{Hunt1966}*{(2.3) and (2.6)}).

The concept that permits a precise definition of the convexity of a space 
was introduced and developed by Turpin in a series of papers 
(cf.\ \cites{Turpin1973a,Turpin1973b}) and a monograph (\cite{Turpin1976}) 
in the early 1970’s. 
Given an $F$-space $\XX$, its galb, denoted by $\GG(\XX)$, 
is the vector space of all sequences $(a_n)_{n=1}^{\infty}$ of scalars 
such that whenever $(x_n)_{n=1}^{\infty}$ is a sequence in $\XX$ 
with $\lim x_n=0$, the series $\sum_{n=1}^{\infty}\ a_n\, x_n$ converges in $\XX$. {We say that a sequence space $\YY$ galbs $\XX$ if $\YY\subseteq\GG(\XX)$. With this terminology, $\XX$ is $p$-convex if and only if $\ell_p\subseteq\GG(\XX)$.}

The galb of certain classical spaces is known. 
Turpin~\cite{Turpin1973a} computed the galb of 
locally bounded, non-locally convex  Orlicz function spaces $L_\varphi(\mu)$, where $\mu$ is either a nonatomic measure or the counting measure, and showed that the result is an Orlicz sequence space $\ell_\phi$ modeled after a different Orlicz function $\phi$.
Hern\'andez \cites{Hernandez1983,Hernandez1984,Hernandez1985} continued the study initiated by Turpin and computed, in particular, the galb of certain vector-valued Orlicz spaces. The study of the convexity of Lorentz spaces took a different route.
Before Turpin invented the notion of galb, Stein and Weiss \cite{SteinWeiss1969} proved that the Orlicz sequence space $\ell \log \ell$ galbs $L_{1,\infty}$, and used this result to achieve a Fourier multiplier theorem for  $L_{1,\infty}$. Sj\"{o}gren \cite{Sjogren1990} concluded the study by (implicitely) proving that $\GG(L_{1,\infty})=\ell \log \ell$.  Later on, the convexity type of Lorentz spaces $L_{1,q}$ for $0<q<\infty$ was estudied (see \cites{Sjogren1992,ColzaniSjogren1999}). In \cite{CCS2007}, general weighted Lorentz spaces were considered.

The geometry of spaces of galbs is quite  unknown, however.
Probably, the most significant advance in this direction since seminal Turpin work was made in \cite{Kalton1977}. Solving a question raised in \cite{Turpin1976}, Kalton proved that if $\XX$ is $p$-convex
 and is not $q$-convex for any $q>p$, then  $\GG(\XX)=\ell_p$.

In this paper, we use galbs to develop a theory of integration for quasi-Banach-space-valued functions  in the spirit of Vogt that fits as well as possible the convexity of the target space.
Our construction is closely related to tensor products, and to carry out it we construct topological tensor products adapted to our neeeds.
More precisely, for an appropriate function quasi-norm $\lambda$ over $\NN$
we define the tensor product space $\XX\otimes_\lambda L_1(\mu)$
so that there are  bounded linear canonical maps
\begin{align*}
J&\colon \XX\otimes_\lambda L_1(\mu) \to  L_1(\mu,\XX), \quad x\otimes f \mapsto x f, \text{ and }\\
I&\colon  \XX\otimes_\lambda L_1(\mu) \to \XX, \quad x\otimes f \mapsto x \int_\Omega f\, d\mu.
\end{align*}
If $I$ factors through $J$, that is, there is a map $\II$ (defined on the range of $J$) such that the diagram
\[\xymatrix{
\XX \otimes_\lambda L_1(\mu) \ar[d]_{J} \ar[drr]^{I}& &\\
L_1^\lambda(\mu,\XX):=J(\XX\otimes_\lambda L_1(\mu)) \ar[rr]_-{\II}& &\XX
}\]
commutes, then  $\II$ defines a suitable 
integral for functions in $L_1^\lambda(\XX)$. Thus, we say that $(\lambda,\XX)$ is amenable if $\lambda$ galbs $\XX$ (i.e., $(a_n)_{n=1}^{\infty}\in\GG(\XX)$ whenever $\lambda((a_n)_{n=1}^{\infty})<\infty$) and $I$ factors through $J$.

There is a tight connection between the existence of the integral $\II$ and 
the injectivity of $J$. In fact, we will prove that  if $(\lambda,\XX)$ is amenable, then $J$ is one-to-one (see Theorem~\ref{thm:one-to-one}).
This connection leads us to study the injectivity of $J$. More generally, we consider the map 
\[
J\colon \XX\otimes_\lambda \LL_\rho\to\LL_\rho(\XX)
\]
associated with the quasi-Banach space $\XX$, the function quasi-norm $\lambda$ and a function quasi-norm $\rho$ over $(\Omega,\Sigma,\mu)$, and we obtain results that generalize those previously obtained for  Lebesgue spaces $L_q(\mu)$ and tensor quasi-norms in the sense of $\ell_p$, $0<p\le q\le \infty$ (see  \cite{Vogt1967}*{Satz 4}).

With the terminology of this paper, Vogt proved that if $\XX$ is a $p$-Banach space, 
$0<p\le 1$, then $(\ell_p,\XX)$ is amenable. 
So, in order to exhibit the applicability of 
the theory of integration developed within this paper, 
we must exhibit new examples of amenable pairs. 
Since the space of galbs of the quasi-Banach space $\XX$  arises from a function quasi-norm on $\NN$, say $\lambda_\XX$, the question of whether the pair $(\lambda_\XX,\XX)$ is amenable arises. For answering it, 
one first need to know whether the space of galbs $\GG(\XX)$ is 
always $1$-concave as a quasi-Banach lattice or not. See Questions~\ref{qt:ame} and \ref{question:B}.  
As long as there is no general answer to these questions, 
we focus on the spaces of galbs that have appeared in the literature. 
In Theorem~\ref{thm:Orlicz} we prove that for all of them 
Question~\ref{qt:ame} has a positive answer. 

Once the theory is built, the first goal should be the study of 
its integration properties. By construction, our integral behaves linearly and has suitable convergence properties. Hence, we finish with a fundamental theorem of calculus 
for our integral (see Theorem~\ref{thm:ftc}). 

The paper is organized as follows. 
In Section~\ref{sec:terminology} we introduce the terminology and notation 
that will be employed. 
The theory of function norms (i.e., the locally convex setting) 
has been deeply developed (cf.\ \cites{LuxZaa1963a,BennettSharpley1988}). 
However, a systematic study in the non-locally convex setting is missing. 
For that reason, in Section~\ref{sect:fqn} we do a brief survey on 
function quasi-norms covering the most relevant aspects, and 
all the results that we need. Section~\ref{sec:galbs} is devoted to galbs.  
In Section~\ref{sec:tensor} we briefly collect some results on tensor products.  
In Section~\ref{sec:integration} we present our main results of integration for 
quasi-Banach-space-valued functions. 
Finally, in Section~\ref{sec:ftc} we give a fundamental theorem of calculus 
that improves \cite{AA2013}*{Theorem 5.2}.

\section{Terminology}\label{sec:terminology}
\noindent
We use standard terminology and notation in Banach space theory as can be found, e.g., in \cites{AlbiacKalton2016}. The unfamiliar reader will find general information about quasi-Banach spaces in \cite{KPR1984}. We next gather the notation on quasi-Banach spaces that we will use.

A \emph{quasi-normed} space will be a vector space over the real or complex field $\FF$ endowed with a \emph{quasi-norm}, i.e., a map $\Vert \cdot\Vert\colon \XX\to [0,\infty)$ satisfying
\begin{enumerate}[label={(Q.\arabic*)}]
\item\label{it:H} $\Vert x\Vert = 0$ if and only if $x=0$;

\item\label{it:Hom} $\Vert t x\Vert =|t| \Vert x\Vert$ for $t \in \FF$ and $x\in \XX$; and

\item\label{it:modconcavity} there is a constant $\kappa\ge 1$ so that for all $x$ and $y$ in $\XX$ we have
\[
\Vert x +y\Vert \le \kappa (\Vert x\Vert +\Vert y\Vert).
\]
\end{enumerate}
The smallest number $\kappa$ in \ref{it:modconcavity} will be called the \emph{modulus of concavity} of the quasi-norm. If it is possible to take $\kappa=1$ we obtain a norm. A quasi-norm clearly defines a metrizable vector topology on $\XX$ whose base of neighborhoods of zero is given by sets of the form $\{x\in \XX \colon \Vert x\Vert<1/n\}$, $n\in \NN$.
Given $0<p\le 1$, a quasi-normed space is said to be \emph{$p$-convex} if it has an absolutely $p$-convex neighborhood of the origin. A quasi-normed space $\XX$ is $p$-convex if and only if there is a constant $C$ such that
\begin{equation}\label{eq:pconvex}
\left\Vert \sum_{j=1}^n x_j \right\Vert^p \le C \sum_{j=1}^n \Vert x_j \Vert^p, \quad n\in\NN, \, x_j\in\XX.
\end{equation}

If, besides \ref{it:H} and \ref{it:Hom}, \eqref{eq:pconvex} holds with $C=1$ we say that $\Vert \cdot\Vert$ is a \emph{$p$-norm}. Any $p$-norm is a quasi-norm with modulus of concavity at most $2^{1/p-1}$. A \emph{$p$-normed} space is a quasi-normed space endowed with a $p$-norm. By the Aoki-Rolewicz theorem \cites{Aoki1942, Rolewicz1957} any quasi-normed space is $p$-convex for some $0<p\le 1$. In turn, any $p$-convex quasi-normed space can be equipped with an equivalent $p$-norm. Hence, any quasi-normed space becomes, for some $0<p\le 1$, a $p$-normed space under suitable renorming.

A \emph{$p$-Banach} (resp.\ \emph{quasi-Banach}) space is a complete $p$-normed (resp.\ quasi-normed) space. It is known that a $p$-convex quasi-normed space is complete if and only if for every sequence $(x_n)_{n=1}^\infty$ in $\XX$ such that $\sum_{n=1}^\infty \Vert x_n\Vert^p<\infty$ the series $\sum_{n=1}^\infty x_n$ converges.

A semi-quasi-norm on a vector space $\XX$ is a map  $\Vert \cdot\Vert\colon \XX\to [0,\infty)$ satisfying \ref{it:Hom} and \ref{it:modconcavity}. A standard procedure, to which we refer as the \emph{completion method} allow us to manufacture a quasi-Banach from a  semi-quasi-norm (see e.g.\ \cite{AACD2018}*{\S2.2}).

As the Hahn-Banach Theorem depends heavily on  convexity, it does not pass through general quasi-Banach spaces. In fact, there are quasi-Banach spaces as $L_p([0,1])$ for $0<p<1$ whose dual space is null (see \cite{Day1940}). Following \cite{KPR1984}, we say that the quasi-Banach space $\XX$ has \emph{point separation property} if for every $f\in\XX\setminus\{0\}$ there is $f^*\in\XX^*$ such that $f^*(f)\not=0$. 

For any subset $A$ of a quasi-Banach space we denote by $[A]$ its closed linear span.

Given a $\sigma$-finite measure space $(\Omega,\Sigma,\mu)$ and a quasi-Banach space $\XX$, we denote by $L_0^+(\mu)$ the set consisting of all measurable functions from $\Omega$ into $[0,\infty]$, and by $L_0(\mu,\XX)$ the vector space consisting of all measurable functions from $\Omega$ into $\XX$. As usual, we identify almost everywhere (a.e.\ for short) coincident functions. We set $L_0(\mu)=L_0(\mu,\FF)$ and
\[
\Sigma(\mu)=\{ A\in\Sigma \colon \mu(A)<\infty\}.
\]
We denote by $\Sp(\mu,\XX)$ the vector space consisting of all integrable 
$\XX$-valued simple functions. That is,
\[
\Sp(\mu,\XX)=[x \chi_E\colon E\in\Sigma(\mu), \ x\in\XX].
\]
We say that $(\Omega,\Sigma,\mu)$ is \emph{infinite-dimensional} if $\Sp(\mu)=\Sp(\mu,\FF)$ is.

Given $f\in L_0^+(\mu)$ we set
\[
\Omega_f(s)=\{ \omega\in\Omega \colon f(\omega)> s\} \text{ and } \rho_f(s)=\rho(\chi_{\Omega_f(s)}), \quad s\in[0,\infty).
\]
Set also $\Omega_f(\infty)=\{ \omega\in\Omega \colon f(\omega)=\infty\}$ and $\rho_f(\infty)=\rho(\chi_{\Omega_f(\infty)})$. If $\rho$ is the function quasi-norm associated with $L_1(\mu)$, then $\mu_f:=\rho_f$ is the \emph{distribution function} of $f$. We say $f$ has a finite distribution function if $\mu_f(s)<\infty$ for all $s>0$. 

An \emph{order ideal} in $L_0(\mu)$ will be a (linear) subspace $L$ of $L_0(\mu)$ such that $\overline{f} \in L$ whenever $f\in L$, and $\max\{f,g\}\in L$ whenever $f$ and $g$ are real-valued functions in $L$. 
A \emph{cone} in $L_0^+(\mu)$ will be a subset $\Cn$ of $L_0^+(\mu)$ such that 
for all $f,g\in\Cn$ and all $\alpha,\beta\ge0$ we have $f<\infty$ a.e., 
$\alpha f+\beta g\in\Cn$, and $\max\{f,g\}\in\Cn$. 
It is immediate that if $L$ is an order ideal in $L_0(\mu)$, then
\[
L^+:=L\cap L_0^+(\mu)
\]
is a cone in $L_0^+(\mu)$; and reciprocally, if $\Cn$ is a cone in $L_0^+(\mu)$, there is a unique order ideal $L$ with $L^+=\Cn$. Namely,
\[
L=\{f\in L_0(\mu) \colon |f|\le g \text{ for some } g\in \Cn\}
\]

Given a quasi-Banach space $\XX$, we say that 
a quasi-Banach space $\UU$ is \emph{complemented} in $\XX$ via a map $S\colon\UU\to\XX$ 
if there is a map $P\colon\XX\to\UU$ such that $P\circ S=\Id_\UU$.

The \emph{unit vector system} is the sequence $(\ee_k)_{k=1}^\infty$ in $\FF^\NN$ defined by $\ee_k=(\delta_{k,n})_{n=1}^\infty$, where $\delta_{k,n}=1$ if $k=n$ and $\delta_{k,n}=0$ otherwise.
A \emph{block basis sequence} with respect to the unit vector system is a sequence $(f_k)_{k=1}^\infty$ such that
\[
f_k=\sum_{n=1+n_{k-1}}^{n_k} a_n \, \ee_n, \quad k\in\NN
\]
for some sequence $(a_n)_{n=1}^\infty$ in $\FF^\NN$ and some increasing sequence $(n_k)_{k=0}^\infty$ of non-negative scalars with $n_0=0$.

\section{Function quasi-norms}\label{sect:fqn}
\noindent
As mentioned in the Introduction, in contrast to the theory of function norms, 
there is no systematic study in the non-locally convex setting. 
In this section we try to go one step forward in that direction. 
We begin with the basic properties of function quasi-norms. 
Here, we do not impose them to satisfy a Fatou property 
(something that Bennet and Sharpley \cite{BennettSharpley1988} do for 
function norms). We devote a subsection to the study of this property. 
Then we study the properties of absolute continuity and domination for function quasi-norms, 
as well as Minkowski-type inequalities. 
We also discuss the use of conditional expectation 
(via the notion of leveling function quasi-norms), 
which will be relevant for the proof of Theorem~\ref{thm:one-to-one}. 
We conclude the section with some comments on function quasi-norms 
over $\NN$ endowed with the counting measure, a specially important particular case. 

\begin{definition}
A \emph{function quasi-norm} over a $\sigma$-finite measure space 
$(\Omega,\Sigma,\mu)$ is a mapping $\rho\colon L_0^+(\mu)\to [0,\infty]$ such that
\begin{enumerate}[label=(F.\arabic*),leftmargin=*,widest=5,series=fqn]
\item\label{FQN:H} $\rho(t f)=t\rho(f)$ for all $t\ge 0$ and $f\in L_0^+(\mu)$;
\item\label{FQN:M} if $f\le g$ a.e., then $\rho(f)\le \rho(g)$;
\item\label{FQN:SF} if $E\in\Sigma(\mu)$, then $\rho(\chi_E)<\infty$;
\item\label{FQN:CM} for every $E\in\Sigma(\mu)$ and every $\varepsilon>0$, there is $\delta>0$ such that $\mu(A)\le \varepsilon$ whenever $A\in\Sigma$ satisfies $A\subseteq E$ and $\rho(\chi_A)\le \delta$; and 
\item\label{FQN:QSA} there is a constant $\kappa$ such that $\rho(f+g)\le\kappa( \rho(f)+\rho(g))$ for all $f$, $g\in L_0^+(\mu)$.
\end{enumerate}
The optimal $\kappa$ in \ref{FQN:QSA} is called the \emph{modulus of concavity} of $\rho$. 
\end{definition}

Notice that \ref{FQN:CM} implies that $\rho(\chi_E)>0$ for all $E\in\Sigma$ with $\mu(E)>0$. 

\begin{definition}
A \emph{function norm} is a function quasi-norm with modulus of concavity $1$. More generally, given $0<p\le 1$, a \emph{function $p$-norm} is a function $\rho\colon L_0^+(\mu)\to [0,\infty]$ which satisfies 
\ref{FQN:H}--\ref{FQN:CM}, 
and
\begin{enumerate}[label=(F.\arabic*),leftmargin=*,widest=5,resume=fqn]
\item\label{FQN:pSA} $\rho^p(f+g)\le \rho^p(f)+\rho^p(g)$ for all $f$, $g\in L_0^+(\mu)$.
\end{enumerate}
\end{definition}
The inequality $a^p+b^p \le 2^{1-p} (a+b)^p$ for all $a$, $b\in[0,\infty]$ and $p\in (0,1]$ yields that any function $p$-norm is a function quasi-norm with modulus of concavity at most $2^{1/p-1}$.

This generalization of the notion of a function norm follows ideas from \cite{BennettSharpley1988} and \cite{LuxZaa1963a}. Asides \ref{FQN:QSA}, the main differences between our definition and that adopted by Luxemburg and Zaanen in \cite{LuxZaa1963a} lie in restricting ourselves to $\sigma$-finite spaces, and in imposing condition~\ref{FQN:SF}, which, on the one hand, prevents from existing non null sets $E$ on which $\rho$ is trivial (in the sense that if $f\in L_0^+(\mu)$ is null outside $E$ then $\rho(f)$ is either $0$ or $\infty$) and, on the other hand, guarantees the existence of enough functions with finite quasi-norm. 
Regarding the approach in \cite{BennettSharpley1988}, we point out that Bennet and Sharpley imposed a function norm to satisfy
\begin{enumerate}[label=(F.\arabic*),leftmargin=*,widest=5,resume=fqn]
\item\label{FQN:LC} for every $E\in\Sigma(\mu)$ there is a constant $C=C_E$ such that
\begin{equation}\label{eq:rhotoL1}
\int_E f\, d\mu \le C_E \rho(f), \quad f\in L_0^+(\mu).
\end{equation}
\end{enumerate}

The most natural examples of functions quasi-norms are $L_p$-quasi-norms, $0<p<\infty$, defined by
\[
f\mapsto\left( \int_\Omega f^p\, d\mu\right)^{1/p} , \quad f\in L_0^+(\mu).
\]
To avoid introducing cumbrous notations, sometimes the symbol $L_p(\mu)$ will mean the function quasi-norm defining the space $L_p(\mu)$ instead of the space itself, and the same convention will be used for Lorentz and Orlicz spaces. Since, if $\mu$ is not purely atomic and $0<p<1$, $L_p(\mu)$ does not satisfy \ref{FQN:LC}, imposing this condition to all function quasi-norms is somewhat nonsense in the non-locally convex setting. Thus we impose its 
 natural substitute \ref{FQN:CM} instead. 
Also, unlike Bennet and Sharpley, we do not a priori impose $\rho$ to satisfy Fatou property (see Section~\ref{sect:Fatou}). 

\begin{definition}
We say that a function quasi-norm $\rho$ is \emph{rearrangement invariant} if every function $f\in L_0^+(\mu)$ with $\rho(f)<\infty$ has a finite distribution function, and $\rho(f)=\rho(g)$ whenever $\mu_f=\mu_g$.
\end{definition}

The proof of the following lemma is based on the elementary inequality
\[
s \rho_f(s)\le  \rho(f), \quad f\in L_0^+(\mu),\quad s\in[0,\infty].
\]
\begin{lemma}\label{lem1}Let $\rho$ be a function quasi-norm over a $\sigma$-finite measure space $(\Omega,\Sigma,\mu)$.
\begin{enumerate}[label=(\roman*),leftmargin=*,widest=iii]
\item\label{lem1:3} 
If $f\in \Sp(\mu)$, then $\rho(|f|)<\infty$.
\item\label{lem1:2} 
If $f\in L_0^+(\mu)$ satisfies $\rho(f)<\infty$, then $f<\infty$ a.e.
\item\label{lem1:1} 
If $f\in L_0^+(\mu)$ satisfies $\rho(f)=0$, then $f=0$ a.e.
\item\label{lem1:4} 
Let $E\in \Sigma(\mu)$, $s>0$, and $\varepsilon>0$.  
Then there is $\delta>0$ such that for all $f\in L_0^+(\mu)$ with $\rho(f)\le \delta$ we have
\[
\mu(\{ \omega\in E \colon f(\omega)> s\}) \le \varepsilon.
\]
\end{enumerate}
\end{lemma}

\begin{proof}
Statement \ref{lem1:3} is clear. 
Now let $f\in L_0^+(\mu)$. If $\rho(f)$ is finite, then $\rho_f(\infty)=0$ and \ref{lem1:2} follows. If $\rho(f)=0$, then $\rho_f(s)=0$ for all $s>0$. Since $\Omega_f(0)=\cup_{n=1}^\infty \Omega_f(2^{-n})$, we obtain \ref{lem1:1}. 
Finally, let $E\in \Sigma(\mu)$, $s>0$, and $\varepsilon>0$. 
By \ref{FQN:CM}, there is $\tilde{\delta}>0$ such that if $A\subseteq E$ 
with $\rho(\chi_A)\le\tilde{\delta}$, then $\mu(A)\le\varepsilon$. 
Take $\delta:=s\tilde{\delta}$, and let $f\in L_0^+(\mu)$ with $\rho(f)\le \delta$. 
Set $A:=\{ \omega\in E \colon f(\omega)> s\}$. 
Since $\rho(\chi_A)\le\rho(f)/s\le\tilde{\delta}$, we obtain \ref{lem1:4}. 
\end{proof}

\begin{definition}
A function quasi-norm $\rho$ is said to be \emph{$p$-convex} if there is a constant $C$ such that
\[
\textstyle
\rho^p(\sum_{j=1}^n f_j)\le C \sum_{j=1}^n \rho^p(f_j), \quad n\in\NN, \, f_j\in L_0^+(\mu).
\]
\end{definition}

\begin{proposition}[Aoki-Rolewicz Theorem for function quasi-norms]\label{prop:AOFQN}Any function quasi-norm is $p$-convex for some $0<p\le 1$. Indeed, if $\kappa$ is the modulus of concavity we can choose $p$ such that $2^{1/p-1}=\kappa$.
\end{proposition}

\begin{proof}
It goes over the lines of the proof of the Aoki-Rolewicz Theorem (see e.g.\ \cite{KPR1984}*{Lemma 1.1}). So, we omit it.
\end{proof}

\begin{definition}
Given two function quasi-norms $\rho$ and $\lambda$ over a $\sigma$-finite measure space $(\Omega,\Sigma,\mu)$, we say that $\rho$ \emph{dominates} $\lambda$ if there is a constant $C$ such that $\lambda(f)\le C \rho(f)$ for all $f\in L_0^+(\mu)$. If $\rho$ dominates and is dominated by $\lambda$, we say that $\rho$ and $\lambda$ are \emph{equivalent}.
\end{definition}

\begin{lemma}\label{lem:2}
Let $0<p\le 1$, and let $\rho$ be a function quasi-norm. 
Then $\rho$ is equivalent to a function $p$-norm if and only if it is $p$-convex.
\end{lemma}

\begin{proof}
It is clear that any function $p$-norm is $p$-convex, and $p$-convexity 
is inherited by passing to an equivalent function quasi norm. 
Reciprocally, if $\rho$ is a $p$-convex function quasi-norm over a $\sigma$-finite measure space $(\Omega,\Sigma,\mu)$, 
then it is immediate that the map map $\lambda\colon L_0^+(\mu)\to[0,\infty]$ given by
\[
\lambda(f)=\inf\left\{ \left( \sum_{j=1}^n \rho^p (f_j) \right)^{1/p} \colon n\in\NN,\ f_j\in L_0^+(\mu), \ f=\sum_{j=1}^n f_j \right\}
\]
is a function $p$-norm equivalent to $\rho$. 
\end{proof}

\begin{corollary}\label{cor:AOR}Any function quasi-norm is equivalent to a function $p$-norm for some $0<p\le 1$.
\end{corollary}

\begin{proof}
It follows from Proposition~\ref{prop:AOFQN} and Lemma~\ref{lem:2}.
\end{proof}

In light of Corollary~\ref{cor:AOR}, it is natural, and convenient in some situations, to restrict ourselves to function quasi-norms that are function $p$-norms for some $p$. However, we emphasize that some $p$-convex spaces arising naturally in Mathematical Analysis are given by a function quasi-norm that is not a  $p$-norm. Take, for instance the $1$-convex (i.e., locally convex) function space $L_{r,\infty}$, $r>1$. So, when working in the general framework of non-locally convex spaces, it is convenient to know whether a given property pass to equivalent function quasi-norms.

\begin{definition}
Let $\rho$ be a function quasi-norm over a $\sigma$-finite measure space $(\Omega,\Sigma,\mu)$, and let $\XX$ be a quasi-Banach space. The space
\[
\LL_\rho(\XX)=\{f\in L_0(\mu,\XX) \colon \Vert f\Vert_\rho:= \rho(\Vert f\Vert)<\infty\}.
\]
endowed with the gauge $\Vert \cdot\Vert_\rho$ will be called the \emph{vector-valued K\"othe space} associated with $\rho$ and $\XX$.
The space $\LL_\rho=\LL_\rho(\FF)$ will be called the \emph{K\"othe space} associated with $\rho$.
\end{definition}
Note that we do not impose the functions in $\LL_\rho(\XX)$ to be strongly measurable. If $\rho$ is the function quasi-norm associated to the Lebesgue space $L_p(\mu)$, $0<p<\infty$, we set $L_p(\mu,\XX):=\LL_\rho(\XX)$. If $A\in\Sigma$, we set $L_p(A,\mu,\XX):=L_p(\mu|_A,\XX)$, where $\mu|_A$ is the restriction of $\mu$ to $\Sigma\cap\Pt(A)$. In general, if $\rho|_A$ is the function quasi-norm defined by $\rho|_A(f)=\rho(\tilde{f})$, where
\[
\tilde{f}(\omega)=\begin{cases} f(\omega) & \text{ if } \omega\in A, \\ 0 & \text{ otherwise,} \end{cases}
\]
we set $\LL_\rho(A,\XX)=\LL_{\rho|_A}(\XX)$.

It is clear that $\LL_\rho$ is an order ideal in $L_0(\mu)$. By Lemma~\ref{lem1}~\ref{lem1:2}, its cone is given by
\[
\LL_\rho^+=\{f\in L_0^+(\mu) \colon \rho(f)<\infty\}.
\]

\begin{lemma}\label{lem2}
Let $\rho$ be a function quasi-norm over a $\sigma$-finite measure space $(\Omega,\Sigma,\mu)$ and $\XX$ be a quasi-Banach space.
\begin{enumerate}[label=(\roman*),leftmargin=*,widest=iii]
\item\label{lem2:1} $\LL_\rho(\XX)$ is a quasi-normed space. 
\item\label{lem2:2} $\Sp(\mu,\XX) \subseteq\LL_\rho(\XX)$.
\item\label{lem2:3} If we endow $L_0(\mu,\XX)$ with the vector topology of the local convergence in measure, then $\LL_\rho(\XX)\subseteq L_0(\mu,\XX)$ continuously.
\item\label{lem2:4} If $\mathcal{K}$ is a closed subset of $\XX$, then $\LL_\rho(\mathcal{K}):=\{ f\in\LL_\rho(\XX) \colon f(\omega)\in \mathcal{K} \text{ a.e.\ } \omega\in\Omega\}$ is closed in $\LL_\rho(\XX)$.
\end{enumerate}
\end{lemma}

\begin{proof}
Statements \ref{lem2:1}, \ref{lem2:2}, and \ref{lem2:3} are straightforward from the very definition of function quasi-norm and Lemma~\ref{lem1}. Now let $\mathcal{K}$ be a closed subset of $\XX$, and let $x$ be a function in $\LL_\rho(\XX)\setminus\LL_\rho(\mathcal{K})$ (assuming that this set is non-empty). There is $\varepsilon>0$ and $A\subseteq\Sigma$ with $\mu(A)>0$ such that $\Vert x(a)-k\Vert \ge \varepsilon$ for all $a\in A$ and all $k\in\mathcal{K}$. Therefore $\Vert x-y\Vert_\rho\ge\varepsilon\rho(\chi_A)>0$ for all $y\in \LL_\rho(\mathcal{K})$, and we obtain \ref{lem2:4}. 
\end{proof}

\begin{lemma}\label{lem:ConABs}
Let $\rho$ be a function quasi-norm, and let $\XX$ be a Banach space. If a sequence $(x_n)_{n=1}^\infty$ converges to $x$ in $\LL_\rho(\XX)$, then $(\lVert x_n \rVert)_{n=1}^\infty$ converges to $\lVert x \rVert$ in $\LL_\rho$.
\end{lemma}

\begin{proof}
It follows from the inequality $\lvert \lVert x_n\rVert -\lVert x\Vert \rvert\le \lVert x_n-x\rVert$ for all $n\in\NN$.
\end{proof}

\begin{proposition}\label{prop21}
Let $\rho$ be a function quasi-norm over a $\sigma$-finite measure space 
$(\Omega,\Sigma,\mu)$, let $\XX$ be a quasi-Banach space, and  
let $(x_n)_{n=1}^\infty$ be a sequence in $L_0(\mu,\XX)$ 
such that $\lim_n \Vert x_n-x\Vert_\rho=0$ for some $x\in L_0(\mu,\XX)$. 
Then, there is a subsequence
$(y_n)_{n=1}^\infty$ of $(x_n)_{n=1}^\infty$ such that $\lim_n y_n=x$ a.e.
\end{proposition}

\begin{proof}
Let $(A_j)_{j=1}^\infty$ be an increasing sequence of finite-measure sets such that $x_n$ is null outside $A=\cup_{j=1}^\infty A_j$ for all $n\in\NN$. Then $\rho(\Vert x\Vert\chi_{\Omega\setminus A})=0$. Therefore $x(\omega)=0$ a.e.\ $\omega\in\Omega\setminus A$. By Lemma~\ref{lem2}~\ref{lem2:3}, for each $j\in\NN$ there is an increasing sequence $(n_k)_{k=1}^\infty$ such that $\lim_k x_{n_k}(\omega)=x(\omega)$ a.e.\ $\omega\in A_j$. The Cantor diagonal technique yields a subsequence $(y_n)_{n=1}^\infty$ of $(x_n)_{n=1}^\infty$ such that $\lim_n y_n(\omega)=x(\omega)$ a.e.\ $\omega\in A$.
\end{proof}

\subsection{The Fatou property}\label{sect:Fatou}
\begin{definition}\label{def:Fatou}
Let $\rho$ be a function quasi-norm over a $\sigma$-finite measure space $(\Omega,\Sigma,\mu)$. We say that $\rho$ has the \emph{rough Fatou property} if there is a constant $C$ such that $\rho(\lim_n f_n)\le C \lim_n \rho(f_n)$ whenever $(f_n)_{n=1}^\infty$ is non-decreasing sequence in $L_0^+(\mu)$. If the above holds with $C=1$ we say that $\rho$ has the \emph{Fatou property}. We say that $\rho$ has the \emph{weak Fatou property} if $\rho(\lim_n f_n)<\infty$ whenever the non-decreasing sequence $(f_n)_{n=1}^\infty$ in $L_0^+(\mu)$ satisfies $\lim_n \rho(f_n)<\infty$.
\end{definition}

Fatou property does not pass to equivalent function quasi-norms. In contrast, both rough and weak Fatou property are preserved. In fact, these two notions are equivalent. 

\begin{proposition}[cf.\ \cite{Amemiya1953}*{Lemma}]\label{prop:roughweakFatou}
If $\rho$ is a function quasi-norm with the weak Fatou property, 
then it also has the rough Fatou property.
\end{proposition}
\begin{proof}
Let $\rho$ be a function quasi-norm over a $\sigma$-finite measure space 
$(\Omega,\Sigma,\mu)$. 
By Corollary~\ref{cor:AOR}, we can assume without loss of generality that it is a function $p$-norm for some $0<p\le 1$. Suppose that $\rho$ does not have the rough Fatou property. Then, for each $k\in\NN$ there is a non-decreasing sequence $(f_{k,n})_{n=1}^\infty$ in $L_0^+(\mu)$ with $\sup_n \rho(f_{k,n})\le 1$ and $\rho(\lim_n f_{k,n})>2^{2k/p}$. The sequence $(g_n)_{n=1}^\infty$ defined by
\[
g_n=\sum_{k=1}^n 2^{-k/p} f_{k,n},\quad n\in\NN,
\]
is non-decreasing, and we have
\[
2^{-k/p} f_{k,n} \le g:=\lim_n g_n, \quad k\le n.
\]
Then $\rho(g)\ge 2^{-k/p}\rho(\lim_n f_{k,n})>2^{k/p}$ for all $k\in\NN$. 
That is, $\rho(g)=\infty$. On the other hand, since $\rho$ is a function $p$-norm, $\rho^p(g_n) \le \sum_{k=1}^n 2^{-k}\le 1$ for all $n\in\NN$. Therefore $\rho$ does not have the weak Fatou property.
\end{proof}

\begin{proposition}[cf.\ \cite{BennettSharpley1988}*{Theorem 1.8}]
Let $\lambda$ and $\rho$ be two function quasi-norms over the same $\sigma$-finite measure space. Suppose that $\rho$ has the weak Fatou property. Then $\rho$ dominates $\lambda$ if and only if $\LL_\rho^+\subseteq \LL_\lambda^+$.
\end{proposition}

\begin{proof}
The direct implication is obvious.
Suppose now that $\rho$ does not dominate $\lambda$. 
Then there is a sequence $(f_n)_{n=1}^\infty$ in $L_0^+(\mu)$ such that $4^{n} \rho(f_n)<\lambda(f_n)$ for all $n\in\NN$. Set
\[
f=\sum_{n=1}^\infty \frac{2^{-n}} {\rho(f_n)} f_n.
\]
Using that $\rho$ has the rough Fatou property (due to Proposition~\ref{prop:roughweakFatou}) and Proposition~\ref{prop:AOFQN}, we obtain that $\rho(f)<\infty$. Since
\[
\lambda(f)\ge \sup_n \frac{ 2^{-n} \lambda(f_n) }{ \rho(f_n) } \ge \sup_n 2^n =\infty, 
\]
the space $\LL_\rho^+$ is not contained in $\LL_\lambda^+$. 
\end{proof}

\begin{definition}
Let $0<p\le 1$ and let $\rho$ be a function quasi-norm over a $\sigma$-finite measure space $(\Omega,\Sigma,\mu)$. We say that $\rho$ has the \emph{Riesz-Fischer $p$-property} if for every sequence $(f_n)_{n=1}^\infty$ in $L_0^+(\mu)$ with $\sum_{n=1}^\infty \rho^p(f_n)<\infty$ we have $\rho(\sum_{n=1}^\infty f_n)<\infty$.
\end{definition}

\begin{lemma}[cf.\ \cite{Amemiya1953}*{Theorem}]
Let $\rho$ be a $p$-convex function quasi-norm with the weak Fatou property. Then $\rho$ has the Riesz-Fischer $p$-property.
\end{lemma}
\begin{proof} Let $(f_n)_{n=1}^\infty$ be a sequence in $L_0^+(\mu)$ with $A:=\sum_{n=1}^\infty \rho^p(f_n)<\infty$. If $C$ denotes the $p$-convexity constant of $\rho$, then 
\[
\rho\left(\sum_{n=1}^m f_n \right)\le C^{1/p} \left(\sum_{n=1}^m \rho^p(f_n)\right)^{1/p} \le C^{1/p}A^{1/p}, \quad m\in\NN.
\]
Hence $\lim_m\rho(\sum_{n=1}^m f_n)<\infty$, and therefore $\rho(\sum_{n=1}^\infty f_n)<\infty$ (since $\rho$ has the weak Fatou property). That is, $\rho$ has the Riesz-Fischer $p$-property.
\end{proof}

\begin{proposition}
Let $0<p\le 1$ and let $\rho$ be a function quasi-norm over a $\sigma$-finite measure space $(\Omega,\Sigma,\mu)$. The following are equivalent.
\begin{enumerate}[label=(\roman*),leftmargin=*,widest=iii]
\item\label{prop:RFC:1} $\rho$ has the Riesz-Fischer $p$-property.
\item\label{prop:RFC:2} There is a constant $C$ such that $\rho^p(\sum_{n=1}^\infty f_n)\le C \sum_{n=1}^\infty \rho^p (f_n)$ for every sequence $(f_n)_{n=1}^\infty$ in $L_0^+(\mu)$.
\item\label{prop:RFC:4} $\LL_\rho(\XX)$ is a quasi-Banach space for any (resp.\ some) nonzero quasi-Banach space $\XX$.
\item\label{prop:RFC:3} $\LL_\rho$ is a $p$-convex quasi-Banach space. 
\end{enumerate}
Moreover, the optimal constant in \ref{prop:RFC:2} is the $p$-convexity constant of $\LL_\rho$. In particular, $\LL_\rho$ is a $p$-Banach space if and only if \ref{prop:RFC:2} holds with $C=1$.
\end{proposition}

\begin{proof}
Assume that \ref{prop:RFC:2} does not hold. Then for every $k\in\NN$ there is a sequence $(f_{k,n})_{n=1}^\infty$ in $L_0^+(\mu)$ such that
\[
\rho^p\left(\sum_{n=1}^\infty f_{k,n} \right)\ge k \text{ and } \sum_{n=1}^\infty \rho^p (f_{k,n}) \le 2^{-k}.
\]
Then $\sum_{(k,n)\in\NN^2} \rho^p (f_{k,n})\le 1$, and also
\[
\rho^p\left(\sum_{(k,n)\in\NN^2} f_{k,n}\right)\ge \rho^p\left(\sum_{n=1}^\infty f_{k,n} \right)\ge k
\]
for all $k\in\NN$. That is, $\rho(\sum_{(k,n)\in\NN^2} f_{k,n})=\infty$. So \ref{prop:RFC:1} does not hold. In other words, \ref{prop:RFC:1} implies \ref{prop:RFC:2}. 

Now assume \ref{prop:RFC:2}. 
Let $\XX$ be a nonzero quasi-Banach space with modulus of concavity $\kappa$. 
By Lemma~\ref{lem2}~\ref{lem2:1}, we already know that $\LL_\rho(\XX)$ is a 
quasi-normed space. Therefore, in order to obtain \ref{prop:RFC:4}, 
it suffices to prove that the series $\sum_{n=1}^{\infty} f_n$ converges in $\LL_\rho(\XX)$ 
for every sequence $(f_n)_{n=1}^{\infty}$ in $\LL_\rho(\XX)$ such that 
\begin{equation}\label{eq:kappaCauchy}
\sum_{n=1}^{\infty} \kappa^{np}\rho^p(\lVert f_n\rVert)<\infty.
\end{equation}
Using \ref{prop:RFC:2} and Lemma~\ref{lem1}~\ref{lem1:2}, we obtain that 
$\sum_{n=1}^{\infty} \kappa^n\lVert f_n\rVert$ converges a.e. in $\Omega$; 
say it converges in $\Omega\setminus\Nt$ where $\mu(\Nt)=0$. 
Set $g_n:=f_n\chi_{\Omega\setminus\Nt}$. 
Obviously $\lVert g_n\rVert\le\lVert f_n\rVert$, 
so $\rho(\lVert g_n\rVert)\le\rho(\lVert f_n\rVert)$ for all $n\in\NN$. 
Then \eqref{eq:kappaCauchy} is also true if we put $g_n$ instead of $f_n$. 

For all $M,N\in\NN$ with $M\ge N$, we have 
$\lVert \sum_{n=N}^{M} g_n\rVert\le \sum_{n=N}^{M} \kappa^n\lVert g_n\rVert$. 
Since $\sum_{n=1}^{\infty} \kappa^n\lVert g_n(t)\rVert$ converges for all $t\in\Omega$, 
$(\sum_{n=1}^{m} g_n(t))_{m=1}^{\infty}$ is a Cauchy sequence in $\XX$.  
Therefore $\sum_{n=1}^{\infty} g_n(t)=:f(t)$ converges for all $t\in\Omega$. 
Let us see that $\sum_{n=1}^{\infty} f_n$ converges to $f$ in $\LL_\rho(\XX)$. 

Notice that if a sequence $(x_n)_{n=1}^{\infty}$ converges to $x$ in $\XX$, 
since $\lVert x\rVert\le \kappa\lVert x_n\rVert + \kappa\lVert x-x_n\rVert$, 
we have $\lVert x\rVert\le \kappa\liminf_n\lVert x_n\rVert$. 
Recall that if two functions $u,v$ in $L_0^+(\mu)$ are equal a.e., 
then $\rho(u)=\rho(v)$. Hence
\begin{align*}
\rho\left(\left\lVert f-\sum_{n=1}^{m} f_n\right\rVert\right) 
&= \rho\left(\left\lVert f-\sum_{n=1}^{m} g_n\right\rVert\right) 
= \rho\left(\left\lVert \sum_{n=m+1}^{\infty} g_n\right\rVert\right)\\
&\le \rho\left(\kappa\liminf_{M\to\infty} \left\lVert \sum_{n=m+1}^{M} g_n\right\rVert \right)
\le \kappa\rho\left( \sum_{n=m+1}^{\infty} \kappa^n\lVert g_n\rVert\right)\\
&\le \kappa \left(\sum_{n=m+1}^{\infty} \kappa^{np}\rho^p(\lVert g_n\rVert)\right)^{1/p}
\xrightarrow[m\to\infty]{} 0. 
\end{align*}
Therefore, we have proved that \ref{prop:RFC:2} implies \ref{prop:RFC:4}. 

Suppose that $\LL_\rho(\XX)$ is a quasi-Banach space 
for some nonzero quasi-Banach space $\XX$. 
Take a nonzero vector $x$ in $\XX$. Since obviously $\FF$ is isomorphic to 
$\{tx\colon t\in\FF\}$, which is a closed subset of $\XX$, 
it follows that $\LL_\rho$ is a quasi-Banach space using Lemma~\ref{lem2}~\ref{lem2:4}. 
By the Aoki-Rolewicz theorem, $\LL_\rho$ is $p$-convex for some $0<p\le 1$. 
Hence \ref{prop:RFC:4} implies \ref{prop:RFC:3}.

Finally, assume that \ref{prop:RFC:3} holds. 
Let $(f_n)_{n=1}^\infty$ be a sequence in $L_0^+(\mu)$ such that 
$\sum_{n=1}^\infty \rho^p(f_n)<\infty$. 
Since $\LL_\rho$ is $p$-convex (with constant $C$), for all $M,N\in\NN$ with $M\ge N$ we have 
\[
\rho\left(\sum_{n=N}^{M} f_n \right) \le C^{1/p} \left(\sum_{n=N}^{M} \rho^p(f_n) \right)^{1/p}.
\]
Therefore $(\sum_{n=1}^{m} f_n)_{m=1}^{\infty}$ is a Cauchy sequence in 
the quasi-Banach space $\LL_\rho(\XX)$, so it converges to a function $f$ in $\LL_\rho(\XX)$. 
By Proposition~\ref{prop21}, there is a subsequence 
$(\sum_{n=1}^{m_j} f_n)_{j=1}^{\infty}$ that converges to $f$ a.e., 
say  in $\Omega\setminus\Nt$ where $\mu(\Nt)=0$. 
Since $(\sum_{n=1}^{m} f_n)_{m=1}^{\infty}$ is non-decreasing, 
it follows that it converges to $f$ in $\Omega\setminus\Nt$. 
That is, $\sum_{n=1}^{\infty} f_n=f$ a.e., 
and therefore $\rho(\sum_{n=1}^{\infty}f_n)=\rho(f)<\infty$. 
Hence \ref{prop:RFC:3} implies \ref{prop:RFC:1}. 
\end{proof}

\subsection{Absolute continuity and domination}\label{subsect:DC}
\begin{definition}
Let $\rho$ be a function quasi-norm over a $\sigma$-finite measure space $(\Omega,\Sigma,\mu)$. We say that $f\in \LL^+_\rho$ is \emph{absolutely continuous} with respect to $\rho$ if
\[
\textstyle
\lim_n \rho(f_n)=\rho(\lim_n f_n)
\]
for every non-increasing sequence $(f_n)_{n=1}^\infty$ in $L_0^+(\mu)$ with $f_1\le f$. If the above holds only in the case when $\lim_n f_n=0$, we say that $f$ is \emph{dominating}.
We denote by $\LL_{\rho}^{a}$ (resp.\ $\LL_{\rho}^{d}$) the set consisting of all $f\in L_0(\mu)$ such that $|f|$ is absolutely continuous (resp.\ dominating). We say that $\rho$ is \emph{absolutely continuous} (resp.\ dominating) if $\LL_\rho^a=\LL_\rho$ (resp.\ $\LL_\rho^d=\LL_\rho$). If $\chi_E\in \LL_{\rho}^{a}$ (resp.\ $\LL_{\rho}^{d}$) for every $E\in\Sigma(\mu)$, we say that $\rho$ is \emph{locally absolutely continuous} (resp.\ locally dominating).
\end{definition}

Notice that domination is preserved under equivalence of function quasi-norms, but absolute continuity is not. Propostion~\ref{prop:TCD} below yields that if the function quasi-norm is continuous (in the sense that $\lim_n \Vert x_n\Vert_\rho=\Vert x\Vert_\rho$ whenever $(x_n)_{n=1}^\infty$ and $x$ in $\LL_\rho$ satisfy $\lim_n \Vert x_n-x\Vert_\rho=0$), then both concepts are equivalent. Notice that any function $p$-norm, $0<p\le 1$, is continuous. So, the existence of non-continuous function quasi-norms is a `pathology' which only occurs in the non-locally convex setting. We must point out that, since it is by no means clear whether absolutely continuous norms are continuous, the terminology could be somewhat confusing. Notwithstanding, we prefer to use terminology similar to that it is customary within framework of function norms.

\begin{proposition}\label{prop:TCD}
Let $\rho$ be a function quasi-norm over a $\sigma$-finite measure space $(\Omega,\Sigma,\mu)$. Suppose that $f\in \LL^+_\rho$ is dominating. Then $\lim_n x_n=x$ in $\LL_\rho(\XX)$ for every quasi-Banach space $\XX$ and every sequence $(x_n)_{n=1}^\infty$ in $L_0(\mu,\XX)$ with $\lim_n x_n=x$ a.e.\ and $\Vert x_n\Vert\le f$ a.e.\ for all $n\in\NN$.
\end{proposition}
\begin{proof}Let $\Nt$ be a null set such that $\sup_n \Vert x_n(\omega)\Vert \le f(\omega)<\infty$ and $\lim_n x_n(\omega)=x(\omega)$ for all $\omega\in\Omega\setminus\Nt$. 
Then $\lVert x(\omega)\rVert\le\kappa f(\omega)$ for all $\omega\in\Omega\setminus\Nt$,
where $\kappa$ is the modulus of concavity of the quasi-norm $\lVert \cdot\rVert$.
Set
\[
f_n=\sup_{j\ge n} \Vert x_j-x\Vert \chi_{\Omega\setminus\Nt}, \quad n\in\NN.
\]
The sequence $(f_n)_{n=1}^\infty$ in $L_0^+(\mu)$ decreases to $0$, and 
$\sup_n f_n\le \kappa(\kappa+1)f$. 
Consequently, $\lim_n \rho(f_n)=0$. 
Since $\Vert x_j -x\Vert_\rho\le \rho(f_n)$ whenever $j\ge n$ we are done.
\end{proof}

\begin{proposition}[cf.\ \cite{BennettSharpley1988}*{Proposition 3.6}]\label{prop:13}
Let $\rho$ be a function quasi-norm over a $\sigma$-finite measure space $(\Omega,\Sigma,\mu)$, and let $f$ be a function in $\LL^+_\rho$. Then, $f$ is dominating if and only if
\[
\lim_n \rho(f\chi_{A_n})=0
\]
whenever the sequence $(A_n)_{n=1}^\infty$ in $\Sigma$ decreases to $\emptyset$.
\end{proposition}

\begin{proof}
The direct implication is obvious. Conversely, suppose that $\rho(f\chi_{A_n})\to 0$ whenever $(A_n)_{n=1}^\infty$ decreases to $\emptyset$. Let $(f_n)_{n=1}^{\infty}$ be a non-increasing sequence of functions in $L_0^+(\mu)$ such that  $f_1\le f$ and $f_n\to 0$. Let us prove that $\rho(f_n)\to 0$. 

Let $\kappa$ be the modulus of concavity of $\rho$, and fix $\varepsilon>0$. 

Assume  first that $\mu(\Omega)<\infty$. Then $\rho(\chi_\Omega)<\infty$, and we can set $s=\varepsilon/(2\kappa\rho(\chi_\Omega))$. For each $n\in\NN$, let $B_n=\{f_n<s\}\subseteq\Omega$. It is a non-decreasing sequence in $\Sigma$ whose union is $\Omega$. Since $f_n \le f\chi_{\Omega\setminus B_n}+s\chi_{B_n}$, we have 
\[
\rho(f_n) \le \kappa\rho(f\chi_{\Omega\setminus B_n}) + \kappa s\rho(\chi_{B_n}) \le \kappa\rho(f\chi_{\Omega\setminus B_n}) + \varepsilon/2 < \varepsilon
\]
for $n$ sufficiently large. 

Now suppose that $\mu(\Omega)=\infty$. Let $(\Omega_m)_{m=1}^{\infty}$ be a non-decreasing sequence in $\Sigma(\mu)$ whose union is $\Omega$. Take $m$ such that $\kappa\rho(f\chi_{\Omega\setminus\Omega_m})<\varepsilon/2$. Since $f_n \le f_n\chi_{\Omega_m}+f\chi_{\Omega\setminus\Omega_m}$, using that $\mu(\Omega_m)<\infty$ and the previous case, we have
\[
\rho(f_n) \le \kappa\rho(f_n\chi_{\Omega_m}) + \kappa \rho(f\chi_{\Omega\setminus\Omega_m}) \le \kappa\rho(f_n\chi_{\Omega_m}) + \varepsilon/2 < \varepsilon
\]
for $n$ sufficiently large. 
\end{proof}

Given a function quasi-norm $\rho$ and a set $E\in\Sigma$ we define
\[
\Phi[E,\rho](t)=\sup\{ \rho(\chi_A) \colon A\in\Sigma,\, A\subseteq E, \mu(A)\le t\},
\]
and we set $\Phi[\rho]=\Phi[\Omega,\rho]$. Notice that the function $\Phi[E,\rho]$ is non-negative and non-decreasing. In particular, there exists the limit of $\Phi[E,\rho](t)$ when $t\to 0^+$. 
\begin{corollary}
A function quasi-norm $\rho$ is locally dominating if and only if $\lim_{t\to 0^+} \Phi[E,\rho](t)=0$ for every $E\in\Sigma(\mu)$.
\end{corollary}

\begin{proof}
If $\lim_{t\to 0^+} \Phi[E,\rho](t)=0$ for every $E\in\Sigma(\mu)$, using Proposition~\ref{prop:13} we obtain that $\rho$ is locally dominating. Now assume that $s:=\lim_{t\to 0^+} \Phi[E,\rho](t)>0$ for some $E\in\Sigma(\mu)$. 
Then there is a sequence $(A_n)_{n=1}^\infty$ of measurable subsets of $E$ such that $\mu(A_n)\le 1/2^n$ and $\rho(\chi_{A_n})>s/2$ for all $n\in\NN$. Set $B_n=\cup_{k=n}^\infty A_k$. The sequence $(B_n)_{n=1}^\infty$ decreases to a null set and $\rho(\chi_{B_n})\ge s/2$ for all $n\in\NN$, so $\chi_{B_1}$ is not dominating. Hence $\rho$ is not locally dominating. 
\end{proof}

\begin{definition}
Let $\rho$ be a function quasi-norm over a $\sigma$-finite measure space $(\Omega,\Sigma,\mu)$. We say say $L\subseteq \LL_\rho$ is an \emph{order ideal with respect to $\rho$} if it is an order ideal and it is closed in $\LL_\rho$.
\end{definition}

\begin{lemma}[cf. \ \cite{BennettSharpley1988}*{Theorem 3.8}]
Let $\rho$ be a function quasi-norm over a $\sigma$-finite measure space $(\Omega,\Sigma,\mu)$. Then $\LL_\rho^d$ is an order ideal with respect to $\rho$.
\end{lemma}

\begin{proof}
It is straightforward that $\LL_\rho^d$ is a subspace of $L_0(\mu)$. 
If a function $f$ belongs to $\LL_\rho^d$, 
obviously $\overline{f}$ also belongs to $\LL_\rho^d$. 
Let $f$ and $g$ be real-valued functions in $\LL_\rho^d$. 
Set $A=\{\omega\in\Omega\colon \lvert f(\omega)\rvert<\lvert g(\omega)\rvert\}$. 
Let $(h_n)_{n=1}^\infty$ be a sequence in $L_0^+(\mu)$ decreasing to $0$ 
with $h_1\le \max\{\lvert f\rvert,\lvert g\rvert\}$. 
Since $\lvert f\rvert$ and $\lvert g\rvert$ are dominating, 
$h_1\chi_{A} \le \lvert g\rvert$, and $h_1\chi_{\Omega\setminus A}\le \lvert f\rvert$, 
we obtain that $\lim_n \rho(h_n\chi_A)=0$ and $\lim_n \rho(h_n\chi_{\Omega\setminus A})=0$.  
Hence $\lim_n \rho(h_n)=0$. Therefore $\max\{\lvert f\rvert,\lvert g\rvert\}$ is 
dominating. This implies that $\lvert \max\{f,g\}\rvert$ is also dominating,  
so $\LL_\rho^d$ is an order ideal. 

Now we prove that $\LL_\rho^d$ is closed in $\LL_\rho$. Let $(f_j)_{j=1}^\infty$ be 
a sequence in $\LL_\rho^{d}$ that converges in $\LL_\rho$ to a function $f$. 
Let $(g_n)_{n=1}^\infty$ be a non-increasing sequence in $L_0^+(\mu)$ with 
$g_1\le \lvert f\rvert$ and $\lim_n g_n=0$. 
Then $g_n\le\min\{ g_n,\lvert f_j\rvert\} +|f-f_j|$ for each $j\in\NN$.
Consequently, if $\kappa$ is the modulus of concavity of $\rho$, we have
\[
\rho(g_n)\le \kappa \rho(\min\{ g_n,\lvert f_j\rvert\}) + \kappa \rho(\lvert f-f_j\rvert).
\]
Hence $\lim_n \rho(g_n)=0$. So $\lvert f\rvert$ is dominating, 
as we wanted to prove.
\end{proof}

\begin{definition}
Let $\rho$ be a function quasi-norm over a $\sigma$-finite measure space $(\Omega,\Sigma,\mu)$. We denote by $\LL_\rho^{b}$ the closure of $\Sp(\mu)$ in $\LL_\rho$.
We say that $\rho$ is \emph{minimal} if $\LL_\rho^b=\LL_\rho$.
\end{definition}

\begin{lemma}[cf. \ \cite{BennettSharpley1988}*{Proposition 3.10}]
Let $\rho$ be a function quasi-norm over a $\sigma$-finite measure space $(\Omega,\Sigma,\mu)$. Then $\LL_\rho^b$ is an order ideal with respect to $\rho$. Moreover $\LL_\rho^{b,+}$
is the closure in $\LL_\rho$ of
\[
\Cn=\{f\in L_0^+(\mu) \colon \Vert f\Vert_\infty<\infty, \quad \mu_f(0)<\infty\}.
\]
\end{lemma}

\begin{proof}
It is obvious that $\LL_\rho^b$ is an order ideal in $L_0(\mu)$, and it is closed in $\LL_\rho$ by definition. Hence $\LL_\rho^b$ is an order ideal with respect to $\rho$. 

Let $f$ be a function in $\Cn$, and set $E:=\{0<f<\infty\}\subseteq\Omega$. 
Since $\mu(E)<\infty$, we have $\rho(\chi_E)<\infty$. 
Fix $\varepsilon>0$, and let $0\le g\le f$ be a simple function such that 
$\lVert f-g\rVert_\infty<\varepsilon/\rho(\chi_E)$. Then 
\[
\rho(f-g) \le \lVert f-g\rVert_\infty \rho(\chi_E) < \varepsilon.
\]
This means that $\Cn$ is contained in $\LL_\rho^{b,+}$. 
Therefore, the closure of $\Cn$ in $\LL_\rho$ is also contained in $\LL_\rho^{b,+}$. 
On the other hand, it is obvious that every non-negative simple function which is finite a.e. 
belongs to $\Cn$. So the second part of the statement follows. 
\end{proof}

\begin{proposition}[cf. \ \cite{BennettSharpley1988}*{Theorem 3.11}]
\label{prop:24}
For any function quasi-norm $\rho$ over a $\sigma$-finite measure space $(\Omega,\Sigma,\mu)$ we have $\LL_\rho^{d}\subseteq \LL_\rho^b$.
\end{proposition}

\begin{proof}
It is enough to prove that $\LL_\rho^{d,+}\subseteq \LL_\rho^{b,+}$.
Let $f$ be a function in $\LL^{d,+}_\rho$. 
Let $(A_n)_{n=1}^\infty$ be an increasing sequence in $\Sigma(\mu)$ 
whose union is $\{f>0\}\subseteq\Omega$. Pick an increasing sequence $(f_j)_{j=1}^\infty$ of measurable positive simple functions with $\lim_n f_n=f$. We have $\lim_n\rho (f-f\chi_{A_n})=0$ and $\lim_j \rho(f\chi_{A_n}-f_j\chi_{A_n})=0$ for each $n\in\NN$. Since $f_j\chi_{A_n}\in\LL_\rho^{b,+}$ for all $j$, $n\in\NN$, we infer that $f\in \LL_\rho^{b,+}$. 
\end{proof}

\begin{corollary}
A function quasi-norm $\rho$ is locally dominating if and only if $\LL_\rho^d= \LL_\rho^{b}$. Moreover if $\rho$ is dominating, then $\rho$ is minimal.
\end{corollary}
\begin{proof}
It is a straightforward consequence of Proposition~\ref{prop:24}
\end{proof}

Since we could need to deal with non-continuous function quasi-norms, we give some results pointing to ensure that $\lim_n \Vert x_n\Vert_\rho=\Vert x\Vert_\rho$ under the assumption that $(x_n)_{n=1}^{\infty}$ converges to $x$.
\begin{lemma}\label{lem:35}
Let $\rho$ be a function quasi-norm over a $\sigma$-finite measure space $(\Omega,\Sigma,\mu)$ with the Fatou property, and let $(f_n)_{n=1}^\infty$ be a sequence in $L_0^+(\mu)$. Then
\[
\rho(\liminf_n f_n)\le \liminf_n \rho(f_n).
\]
\end{lemma}

\begin{proof}
Just apply Fatou property to $\inf_{k\ge n} f_k$, $n\in\NN$.
\end{proof}

\begin{lemma}\label{lem:34}
Let $\rho$ be a function quasi-norm with the Fatou property  
and $\XX$ be a Banach space. 
If $x\in\LL_\rho(\XX)$ and $(x_n)_{n=1}^\infty\subseteq\LL_\rho(\XX)$ satisfy
$\sup_n \Vert x_n\Vert \le \Vert x\Vert$ and $\lim_n \rho(\Vert x_n-x\Vert)=0$,  
then $\lim_n \rho(\Vert x_n\Vert)=\rho(\Vert x\Vert)$.
\end{lemma}

\begin{proof}
Obviously, $\limsup_n\rho(\lVert x_n\rVert)\le\rho(\lVert x\rVert)$. 
Let us see now that $\rho(\lVert x\rVert)\le\liminf_n\rho(\lVert x_n\rVert)$. 
Let $(y_n)_{n=1}^{\infty}$ be a subsequence of $(x_n)_{n=1}^{\infty}$ 
such that $\lim_n\rho(\lVert y_n\rVert)=\liminf_n\rho(\lVert x_n\rVert)$. 
Since $\lim_n\rho(\lVert x-y_n\rVert)=0$, by Lemma~\ref{lem:ConABs} we have 
$\lim_n\rho(\lVert x\rVert-\lVert y_n\rVert)=0$. Then Proposition~\ref{prop21}  
guarantees the existence of a subsequence $(z_n)_{n=1}^{\infty}$ of $(y_n)_{n=1}^{\infty}$ 
such that $\lim_n\lVert z_n\rVert=\lVert x\rVert$. Using Lemma~\ref{lem:35} we obtain
\[
\rho(\lVert x\rVert)=\rho(\lim_n\lVert z_n\rVert)\le\liminf_n\rho(\lVert z_n\rVert)
=\lim_n\rho(\lVert y_n\rVert),
\]
as we wanted to prove. 
\end{proof}

\begin{lemma}\label{lem:39}
Let $\rho$ be a function quasi-norm over a $\sigma$-finite measure space $(\Omega,\Sigma,\mu)$ with the Fatou property, and let $\XX$ be a Banach space. 
If $x\in L_0(\mu,\XX)$ and $(x_n)_{n=1}^\infty\subseteq L_0(\mu,\XX)$ satisfy 
$\lim_n x_n=x$ a.e., and $\sup_n \Vert x_n\Vert\le g$ for some $g\in \LL^{a,+}_\rho$,  
then $\lim_n \rho(\Vert x_n\Vert)=\rho(\Vert x\Vert)$.
\end{lemma}

\begin{proof}
Note that since $\lim_n x_n=x$ a.e.\ and $\XX$ is a Banach space, we have 
$\lim_n \lVert x_n\rVert=\lVert x\rVert$ a.e. 
Consider two particular cases. 
First, suppose that $\lVert x_n\rVert\le\lVert x\rVert$ for all $n\in\NN$. 
Obviously, $\limsup_n\rho(\lVert x_n\rVert)\le\rho(\lVert x\rVert)$. 
Then, by Lemma~\ref{lem:35}, $\rho(\lVert x\rVert)\le\liminf_n\rho(\lVert x_n\rVert)$. 
Second, suppose that $\lVert x_n\rVert\ge\lVert x\rVert$ for all $n\in\NN$. 
Obviously, $\liminf_n\rho(\lVert x_n\rVert)\ge\rho(\lVert x\rVert)$. 
Set $g_n=\sup_{k\ge n} \Vert x_k\Vert$. Then $g\ge g_1$ 
and $(g_n)_{n=1}^{\infty}$ is non-increasing with $\lim_ng_n=\lVert x\rVert$ a.e. 
Using the absolute continuity of $g$, we have 
$\limsup_n\rho(\lVert x_n\rVert)\le\lim\rho(g_n)=\rho(\lVert x\rVert)$. 
In the general case, set $g_n=\min\{\Vert x_n\Vert ,\Vert x\Vert\}$ 
and $h_n= \max\{\Vert x_n\Vert,\Vert x\Vert\}$. 
Then both $(\rho(g_n))_{n=1}^{\infty}$ and $(\rho(h_n))_{n=1}^{\infty}$ 
converge to $\rho(\lVert x\rVert)$. 
Since $g_n\le\lVert x_n\rVert\le h_n$, the statement follows. 
\end{proof}

\begin{proposition}
Let $\rho$ be a function quasi-norm over a $\sigma$-finite measure space $(\Omega,\Sigma,\mu)$ with the Fatou property, and let $\XX$ be a Banach space. 
If $x\in L_0(\mu,\XX)$ and $(x_n)_{n=1}^\infty\subseteq L_0(\mu,\XX)$ satisfy 
$\lim_n\lVert x-x_n\rVert_\rho=0$, and $\sup_n \Vert x_n\Vert\le g$ 
for some $g\in \LL^{a,+}_\rho$,  then $\lim_n \rho(\Vert x_n\Vert)=\rho(\Vert x\Vert)$.
\end{proposition}

\begin{proof}
It suffices to prove that any subsequence of $(x_n)_{n=1}^\infty$ has a further subsequence $(y_n)_{n=1}^\infty$ with $\lim_n \Vert y_n\Vert_\rho=\Vert x\Vert_\rho$. But this follows combining Proposition~\ref{prop21} with Lemma~\ref{lem:39}.
\end{proof}

\subsection{The role of lattice convexity and Minkowski-type inequalities}
Function spaces built from function quasi-norms have a lattice structure. 
Let $\rho$ be a function quasi-norm over a $\sigma$-finite measure space 
$(\Omega,\Sigma,\mu)$. 
Given $0<p\le \infty$, we say that $\rho$ is \emph{lattice $p$-convex }(resp.\ \emph{concave}) if $\LL_\rho$ is. Equivalently, $\rho$ is lattice $p$-convex (resp.\ concave) if and only if there is a constant $C$ such that $G \le C H$ (resp.\ $H\le CG$) for every $n\in\NN$ and $(f_j)_{j=1}^n$ in $L_0^+(\mu)$, where
\[
\textstyle
G=\rho((\sum_{j=1}^n f_j^p)^{1/p}), \quad H=(\sum_{j=1}^n \rho^p(f_j))^{1/p}.
\]
If the above holds for disjointly supported families, we say that $\rho$ satisfies an \emph{upper} (resp.\ \emph{lower}) $p$-estimate.

If $\rho$ is lattice $p$-convex, then it is $\overline p$-convex, where $\overline p=\min\{1,p\}$. The notions of $1$-convexity and lattice $1$-convexity are equivalent.
This identification does not extend to $p<1$ since there are function quasi-norms over $\NN$ which are lattice $p$-convex for no $p>0$ (see \cite{Kalton1986}). Kalton \cite{Kalton1984b} characterized quasi-Banach lattices (in particular, function quasi-norms) that are $p$-convex for some $p$ as those that are $L$-convex. We say that a function quasi-norm is \emph{$L$-convex} if there is $0<\varepsilon<1$ such that if $f$ and $(f_j)_{j=1}^n$ in $L_0^+(\mu)$ satisfy
\[
\max_{1\le j \le n} f_j \le f \quad \text{ and } \quad \frac{1}{n}\sum_{j=1}^n f_j \ge (1-\varepsilon)f,
\]
then $\max_{1\le j \le n} \rho(f_j)\ge \varepsilon \rho( f)$.

Given $0<r<\infty$, the \emph{$r$-convexified quasi-norm} $\rho^{(r)}$ is defined by
\[
\rho^{(r)}(f) =\rho^{1/r}( f^r ).
\]
It is straightforward to check that $\rho^{(r)}$ is a function quasi-norm. If $\rho$ has the Fatou (resp.\ weak Fatou) property, then $\rho^{(r)}$ does have. If $\rho$ is $p$-convex (resp.\ concave), then $\rho^{(r)}$ is $pr$-convex (resp.\ concave).
We set
\[
\LL_\rho^{(r)} = \LL_{\rho^{(r)}}.
\]
A question implicit in Section~\ref{subsect:DC} is whether any $p$-convex function quasi-norm with the weak Fatou property is equivalent to a function $p$-norm with the Fatou property. For function norms the answer to this question is positive, and its proof relies on using the associated gauge $\rho'$ given by
\[
\rho'(f)=\sup\left\{ \int_\Omega f g\, d\mu \colon g\in L_0^+(\mu), \ \rho(g)\le 1\right\}.
\]
In fact, we have the following.
\begin{lemma}[see \cite{BennettSharpley1988}*{Theorem 2.2}]
Let $\rho$ be a function quasi-norm fulfiling \ref{FQN:LC}. Then $\rho'$ is a function norm with the Fatou property.
\end{lemma}

\begin{proof}
It is a routine checking.
\end{proof}

\begin{theorem}[cf.\, \cite{BennettSharpley1988} and \cite{Zaanen1983}*{Theorem 112.2}]\label{thm:dual}
Let $\rho$ be a function norm with the weak Fatou property. Suppose that $\rho$ satisfies \ref{FQN:LC}. Then $\rho''$ is equivalent to $\rho$. Moreover, if $\rho$ has the Fatou property, then $\rho''=\rho$.
\end{theorem}

In the non-locally convex setting, it is hopeless to try to obtain full information for $\rho$ from the associated function norm $\rho'$. Nonetheless, the following is a partial positive answer to the aforementoned question.

\begin{proposition}\label{prop:73}
Let $0<p<\infty$ and let $\rho$ be a function quasi-norm over a $\sigma$-finite measure space $(\Omega,\Sigma,\mu)$. Suppose that $\rho$ is $p$-convex, has the weak Fatou property, and that for every $E\in\Sigma(\mu)$ there is a constant $C_E$ such that $\int_E f^{p}\, d\mu \le C_E \rho(f)$ for all $f\in L_0^+(\mu)$.
Then $\rho$ is equivalent to a function $p$-norm with the Fatou property. In fact, there is $G\subset L_0^+(\mu)$ such that $\rho$ is equivalent to the function quasi-norm $\lambda$ given by
\[
\lambda(f)=\sup_{g\in G} \left(\int_\Omega f^p g \, d\mu\right)^{1/p}.
\]
\end{proposition}

\begin{proof}
The function quasi-norm $\rho^{(1/p)}$ is $1$-convex and, then, equivalent to a function norm $\sigma$. The properties of $\rho$ yields that $\sigma$ satisfies \ref{FQN:LC} and has the weak Fatou property. By Theorem~\ref{thm:dual}, $\sigma$ is equivalent to the function norm $\sigma''$. Consequently, $\rho$ is equivalent to the function quasi-norm $\sigma''^{(p)}$.
\end{proof}

\begin{definition}
Let $\rho$ be a function quasi-norm over a $\sigma$-finite measure space 
$(\Omega,\Sigma,\mu)$, and let $(\Theta,\Ts,\nu)$ be another $\sigma$-finite measure space. 
Given $f\in L_0^+(\mu\otimes\nu)$ and $g\in L_0^+(\nu\otimes\mu)$ we set
\begin{align*}
\rho[1,f] &\colon \Theta \to [0,\infty], \quad 
\rho[1,f] (\theta) = \rho(f(\cdot,\theta)); \text{ and }\\
\rho[2,g] &\colon \Theta \to [0,\infty], \quad 
\rho[2,g](\theta) = \rho (g(\theta,\cdot)).
\end{align*}
\end{definition}

\begin{proposition}\label{prop:ms}
Let $\rho$ be a locally absolutely continuous function quasi-norm 
over a $\sigma$-finite measure space $(\Omega,\Sigma,\mu)$ with the Fatou property. 
Let $(\Theta,\Ts,\nu)$ be another $\sigma$-finite measure space. 
Let $f\in L_0^+(\mu\otimes\nu)$ and $g\in L_0^+(\nu\otimes\mu)$. 
Then $\rho[1,f]$ and $\rho[2,g]$ are measurable functions.
\end{proposition}

\begin{proof}
It suffices to prove the result for $f$. The Fatou property yields that if the result holds for a non-decreasing sequence $(f_n)_{n=1}^\infty$, then it also holds for $\lim_n f_n$. Consequently, we can suppose that $\mu(\Omega)<\infty$ and that $f$ is a measurable simple function. Given a measurable simple positive function $f$ we denote by $\Mt_f$ the set consisting of all $E$ in the product $\sigma$-algebra $\Sigma\otimes\Ts$ such that the result holds for $f+t\chi_E$ for every $t\ge 0$. The absolute continuity and the Fatou property yields that $\Mt_f$ is a monotone class for any measurable simple function $f$. Therefore, if $\RRt$ denotes the algebra consisting of all finite disjoint unions of measurable rectangles, the monotone class theorem yields that $\RRt\subseteq \Mt_f$ implies $\Sigma\otimes\Ts\subseteq\Mt_f$. Let $\Cn_r$ denote the cone consisting of all positive functions measurable with respect to $\RRt$. Given $n\in\NN$, let $\Cn[n]$ be the cone consisting of all measurable non-negative functions which take at most $n-1$ different positive values. It is straightforward to check that the result holds for all functions in $\Cn_r=\Cn_r+\Cn[1]$. Suppose that the result holds for all functions in $\Cn_r+\Cn[n]$. Then $\RRt\subseteq \Mt_f$ for all $f\in \Cn_r+\Cn[n]$. Consequently, $\Sigma\otimes\Ts\subseteq\Mt_f$ for all $f\in \Cn_r+\Cn[n]$. In other words, the result holds for all functions in $\Cn_r+\Cn[n+1]$. By induction, the result holds for every $f\in \Cn:=\cup_{n=1}^\infty \Cn_r+\Cn[n]$. Since $\Cn$ is the cone consisting of all measurable simple non-negative functions, we are done.
\end{proof}

Proposition~\ref{prop:ms} allows us to iteratively apply function quasi-noms to measurable functions defined on product spaces. A \emph{Minkowski-type inequality} is an inequality that compares the gauges that appear when iterating in different ways.

\begin{definition}
Let $\rho$ and $\lambda$ be locally absolutely continuous function quasi-norms with the Fatou property over $\sigma$-finite measure spaces $(\Omega,\Sigma,\mu)$ and $(\Theta, \Ts,\nu)$ respectively. Given $f\in L_0^+(\mu\otimes\nu)$ we set
\[
(\rho,\lambda)[1,2](f)= \rho(\lambda[2,f]), \quad (\lambda,\rho)[2,1](f)= \lambda(\rho[1,f]).
\]
We say that the pair $(\rho,\lambda)$ has the \emph{Minkowski's integral inequality} (MII for short) \emph{property} if there is a constant $C$ such that
\[
(\rho,\lambda)[1,2](f) \le C (\lambda,\rho)[2,1](f)
\]
for all $f\in L_0^+(\mu\otimes\nu)$.
\end{definition}

The following result is obtained from the corresponding one for function norms \cite{Schep1995}. We do not know whether a direct proof which circumvent using lattice convexity is possible.
\begin{theorem}\label{thm:MIIpconvex}
Let $\rho$ and $\lambda$ be locally absolutely continuous $L$-convex function quasi-norms with the Fatou property. Then $(\rho,\lambda)$ has the MII property if and only if there is $0<p\le\infty$ such that $\lambda$ is lattice $p$-convex and $\rho$ is lattice $p$-concave.
\end{theorem}

\begin{proof}
Pick $0<s<\infty$ such that $\rho^{(s)}$ and $\lambda^{(s)}$ are $1$-convex. Since
\[
(\rho^{(s)},\lambda^{(s)})[1,2](f)=\left( (\rho,\lambda)[1,2](f^s)\right)^{1/s},
\]
$(\rho,\lambda)$ has the MII property if and only if $(\rho^{(s)},\lambda^{(s)})$ does have. It turn, by \cite{Schep1995}*{Theorems 2.3 and 2.5}, $(\rho^{(s)},\lambda^{(s)})$ has the MII property if and only if there is $q\in(0,\infty]$ such that $\lambda^{(s)}$ is lattice $q$-convex and $\rho^{(s)}$ is lattice $q$-concave. This latter condition is equivalent to the existence of $p\in(0,\infty]$ (related with $q$ by $q=sp$) as desired.
\end{proof}

Given $0<p<\infty$ and a $\sigma$-finite measure space $(\Omega,\Sigma,\mu)$, the Lebesgue space $L_p(\mu)$ is absolutely continuous and lattice $p$-convex. 
Moreover, if $\mu$ is infinite-dimensional, then $L_p(\mu)$ is not lattice $q$-concave for any $q<p$. Consequently, we have the following.
\begin{proposition}Let $0<p<\infty$ and $\rho$ be a locally absolutely continuous $L$-convex function quasi-norm over an infinite-dimensional $\sigma$-finite measure space. 
Given another $\sigma$-finite measure space $(\Omega,\Sigma,\mu)$ such that $L_0(\mu)$ is infinite-dimensional, the pair $(\rho,L_p(\mu))$ has the MII property if and only if $\rho$ is $p$-concave.
\end{proposition}

Another K\"othe space of interest for us is the weak Lorentz space $L_{1,\infty}(\mu)$ defined from the function quasi-norm
\[
f\mapsto\sup_{s>0} s \mu_f(s)=\sup_{s>0} s \mu\{\omega\in\Omega \colon f(\omega)\ge s\}, \quad f\in L_0^+(\mu).
\]
We will denote by $\Vert \cdot\Vert_{1,\infty}$ the quasi-norm in $L_{1,\infty}(\mu)$. We infer from the properties of the distribution function that $L_{1,\infty}(\mu)$ is continuous, has the Fatou property, and it is locally dominating. Kalton \cite{Kalton1980b} proved that then $L_{1,\infty}([0,1])$ is lattice $p$-convex for any $p<1$. We emphasize that the milestone paper \cite{Kalton1984b} allows to achieve this convexity result regardless the $\sigma$-finite measure space $(\Omega,\Sigma,\mu)$. In fact, given $0<p<1$, the $p^{-1/2}$-convexified of $L_{1,\infty}(\mu)$, namely the Lorentz space $L_{p^{-1/2},\infty}$, is locally convex \cite{Hunt1966}. Therefore, by \cite{Kalton1984b}*{Theorem 2.2}, $L_{p^{-1/2},\infty}$ is lattice $p^{1/2}$-convex. Consequently, $L_{1,\infty}(\mu)$ is lattice $p$-convex. Since $L_{1,\infty}(\mu)$ is not locally convex unless finite-dimensional \cite{Hunt1966}, we have the following.

\begin{theorem}\label{thm:WL}
Let $\rho$ be a locally absolutely continuous $L$-convex function quasi-norm, 
and let $(\Omega,\Sigma,\mu)$ an infinite-dimensional $\sigma$-finite measure space. 
Then $(\rho,L_{1,\infty}(\mu))$ has the MII property if and only if 
$\rho$ is $p$-concave for some $p<1$.
\end{theorem}

\subsection{Conditional expectation in quasi-Banach function spaces}
Given a sub-$\sigma$-algebra $\Sigma_0\subseteq\Sigma$, we denote by $L_0^+(\mu,\Sigma_0)$ the set consisting of all non-negative $\Sigma_0$-measurable functions. Given $f\in L_0^+(\mu)$ there is a unique $g\in L_0^+(\mu,\Sigma_0)$ such that $\int_A f\, d\mu=\int_A g \, d\mu$ for all $A\in\Sigma_0$. We say that $g$ is the \emph{conditional expectation} of $f$ with respect to $\Sigma_0$, and we denote $\EE(f,\Sigma_0):=g$.

\begin{definition}
Let $\rho$ be a function quasi-norm over a $\sigma$-finite measure space $(\Omega,\Sigma,\mu)$. We say that $\rho$ is \emph{leveling} if there is a constant $C$ such that $\rho(\EE(f,\Sigma_0))\le C \rho(f)$ for every finite sub-$\sigma$-algebra $\Sigma_0$ and every $f\in L_0^+(\mu)$.
\end{definition}

This terminology follows that used in \cite{EH1953}. We remark that Ellis and Halperin imposed leveling function norms to satisfy the above definition with $C=1$. Not imposing conditional expectations to be contractive turns the notion stable under equivalence.

Given a function quasi-norm $\rho$, a sub-$\sigma$-algebra $\Sigma_0$, and a quasi-Banach space $\XX$, we denote by $\LL_\rho(\Sigma_0,\XX)$ the space consisting of all $\Sigma_0$-measurable functions in $\LL_\rho(\XX)$. Note that, if $\rho|_{\Sigma_0}$ is the restriction of $\rho$ to $\Sigma_0$, then $\LL_\rho(\Sigma_0,\XX)=\LL_{\rho|_{\Sigma_0}}(\XX)$. For further reference, we write down an elementary result.

\begin{lemma}\label{lem:CE}
Let $\rho$ be a leveling function quasi-norm over a $\sigma$-finite measure space 
$(\Omega,\Sigma,\mu)$. Then there is a constant $C$ such that for any 
finite sub-$\sigma$-algebra $\Sigma_0$ there is positive projection 
$T\colon \LL_\rho\to \LL_\rho(\Sigma_0)$ such that $\Vert T\Vert\le C$ 
and $\int_A f\, d\mu=\int_A T(f)\, d\mu$ whenever $f\ge 0$ or $\int_A |f|\, d\mu<\infty$.
\end{lemma}

\begin{definition}
If $\rho$, $\Sigma_0$ and $T$ are as in Lemma~\ref{lem:CE}, we denote $\EE[\rho,\Sigma_0]:=T$.
\end{definition}

\begin{lemma}\label{lem:LevInt}
Leveling function quasi-norms satisfy \ref{FQN:LC}.
\end{lemma}

\begin{proof}
Let $\rho$ be a leveling function quasi-norm over a $\sigma$-finite measure space 
$(\Omega,\Sigma,\mu)$.
Given $A\in\Sigma(\mu)$ with $\mu(A)>0$, let $\Sigma_0$ be the smallest $\sigma$-algebra containing $A$. For all $f\in L_0^+(\mu)$ we have
\[
\int_A f\, d\mu \le \frac{\mu(A)}{\rho(\chi_A)} \rho (\EE(f,\Sigma_0)) \le C \frac{\mu(A)}{\rho(\chi_A)} \rho(f).\qedhere
\]
\end{proof}

It is known that, if $q\ge 1$, $L_q(\mu)$ has the conditional expectation property. Locally convex Lorentz and Orlicz spaces do have. More generally, we have the following. Recall that a measure space is said to be \emph{resonant} if either is non-atomic or it consists of equi-measurable atoms.

\begin{theorem}
Let $\rho$ be a rearrangement invariant function norm over a resonant measure space. If $\rho$ satisfies \ref{FQN:LC}, then it is leveling.
\end{theorem}

\begin{proof}
By Calder\'on-Mitjagin Theorem (see \cites{Calderon1966, Mitjagin1965}, and also \cite{BennettSharpley1988}*{Theorem~2.2}), $\LL_\rho$ is an interpolation space between $L_1$ and $L_\infty$. Since both $L_1$ and $L_\infty$ are leveling, the result follows by interpolation.
\end{proof}

\subsection{Function quasi-norms over $\NN$}
Suppose that $\rho$ is a function quasi-norm over $\NN$ endowed with the counting measure. In this particular case, $\rho$ is locally dominating, and the space of integrable simple functions is the space $c_{00}$ consisting of all eventually null sequences. Concerning the density of $c_{00}$ in $\LL_\rho$ we have the following.

\begin{proposition}\label{prop:looSE}
Let $\rho$ be a function quasi-norm over $\NN$. Then $\rho$ is not minimal if and only if $\ell_\infty$ is a subspace of $\LL_\rho$, in which case $\LL_\rho$ has block basic sequence equivalent to the unit vector system of $\ell_\infty$.
\end{proposition}

Before tackling the proof of Proposition~\ref{prop:looSE} we give an auxiliary lemma that will be used a couple of times.

\begin{lemma}\label{lem:NS}
Let $\rho$ be a function quasi-norm over $\NN$ and 
let $(a_n)_{n=1}^\infty$ be a sequence in $\LL_\rho$. 
Then $(a_n)_{n=1}^\infty$ does not belong to $\LL^b_\rho$ if and only 
there is an increasing sequence $(m_k)_{k=1}^\infty$ of non-negative integers such that
\[
\textstyle
\inf_{k\in\NN} \rho ( ( |a_n|)_{n=1+m_{2k-1}}^{m_{2k}} )>0.
\]
\end{lemma}
\begin{proof}
Use that $(a_n)_{n=1}^\infty\in\LL_\rho\setminus\LL^b_\rho$ if and only if the series $\sum_{n=1}^\infty a_n\, \ee_n$ does not converge.
\end{proof}

\begin{proof}[Proof of Proposition~\ref{prop:looSE}]
Assume that $\LL_\rho^b\not=\LL_\rho$. By Lemma~\ref{lem:NS}, there is $(a_n)_{n=1}^\infty$ in $[0,\infty)^\NN$ such that, if
\[
\textstyle
x_k= \sum_{n=1+m_{2k-1}}^{m_{2k}} a_n\, \ee_n, \quad k\in\NN,
\]
then $\inf_k \Vert x_k\Vert_\rho>0$ and $\sup_m \Vert \sum_{k=1}^m x_k\Vert<\infty$. So, $(x_k)_{k=1}^\infty$ is a block basic sequence as desired.
\end{proof}

\begin{corollary}\label{cor:96}
Let $\rho$ be a function quasi-norm over $\NN$. If $\rho$ satisfies a lower $p$-estimate for some $p<\infty$, then $\rho$ is minimal and $L$-convex.
\end{corollary}

\begin{proof}
Our assumptions yields that $\ell_\infty$ is not finitely represented in $\LL_\rho$ by means of block basic sequences. Then, result follows from Proposition~\ref{prop:looSE} and \cite{Kalton1984b}*{Theorem 4.1}.
\end{proof}

Notice that function quasi-norms over $\NN$ are closely related to unconditional bases. In fact, if $\rho$ is a function quasi-norm over $\NN$, then the unit vector system $(\ee_n)_{n=1}^\infty$ is an unconditional basis of $\LL_\rho^b$. Reciprocally, if $(\xx_n)_{n=1}^\infty$ is an unconditional basis of a quasi-Banach space $\XX$, then the mapping
\[
\rho\left(( a_n)_{n=1}^\infty\right)=\sup \left\{ \left\Vert \sum_{n=1}^\infty b_n\, \xx_n\right\Vert \colon (b_n)_{n=1}^\infty\in c_{00},\, \forall \, n\in\NN \; |b_n|\le |a_n|\right\}
\]
defines a function quasi-norm over $\NN$, and the linear map given by $\xx_n\mapsto\ee_n$ extends to an isomorphism from $\XX$ onto  $\LL_\rho^b$.

\section{The galb of a quasi-Banach space}\label{sec:galbs}
\noindent
In this section we deal with function quasi-norms associated with galbs of quasi-Banach spaces.

\begin{definition}
A function quasi-norm over $\NN$ is said to be \emph{symmetric} (or \emph{rearrangement invariant}) if $\rho(f)=\rho(g)$ whenever $g=(b_n)_{n=1}^\infty$ is a rearrangement of $f=(a_n)_{n=1}^\infty$, i.e., there is a permutation $\pi$ of $\NN$ such that $b_n=a_{\pi(n)}$ for all $n\in\NN$.
\end{definition}
The symmetry of $\rho$ allows us to safely define $\rho(f)$ for any countable family of non-negative scalars $f=(a_j)_{j\in J}$. In the language of bases, if $\rho$ is a symmetric function-quasi-norm, then the unit vector system is a $1$-symmetric basis of $\LL_\rho^b$.

\begin{definition}
Given a quasi-Banach space $\XX$ and a sequence $f=(a_n)_{n=1}^\infty$ in $ [0,\infty]^\NN$ we define
\[
\lambda_\XX(f) =\sup\left\{\left\Vert \sum_{n=1}^{N} a _{n}x_{n} \right\Vert \colon N\in\NN, \Vert x_n\Vert \le 1\right\}
\]
if $a_n<\infty$ for all $n\in\NN$, and $\lambda_\XX(f)=\infty$ otherwise.
\end{definition}

\begin{proposition}\label{prop:FQG}
Let $\XX$ be a quasi-Banach space. Then $\lambda_\XX$ is a symmetric function quasi-norm with modulus of concavity at most that of $\XX$. Moreover,

\begin{enumerate}[label=(\roman*),leftmargin=*,widest=iv]
\item $\lambda_\XX$ is locally absolutely continuous.
\item $\lambda_\XX$ has the Fatou property.
\item If $\YY$ is a subspace of $\XX$, then $\lambda_\XX$ dominates $\lambda_\YY$.
\item If $\XX$ and $\YY$ are isomorphic, then $\lambda_\XX$ and $\rho_\YY$ are equivalent.
\item $(\lambda_\XX,\lambda_\XX)[1,2]$ dominates $\lambda_\XX$ (regarded as a function quasi-norm over $\NN^2$).
\item If $\XX$ is a $p$-Banach space, $0<p\le 1$, then $\lambda_\XX$ is a function $p$-norm.
\item\label{galb:latticecon} If $\XX$ a $p$-convex quasi-Banach lattice, $0<p\le 1$, then $\lambda_\XX$ is lattice $p$-convex.
\end{enumerate}
\end{proposition}
 
\begin{proof}
We will prove \ref{galb:latticecon}, and we will leave the other assertions, which are reformulations of results from \cite{Turpin1976}, as an exercise for the reader.
Notice that  $\ell_1$ is a $p$-convex lattice, that is, we have
\[
\sum_{n=1}^\infty \left( \sum_{j=1}^J |a_{n,j}|^p\right)^{1/p} \le\left( \sum_{j\in J}\Big( \sum_{n=1}^\infty |a_{n,j}|\Big)^{p}\right)^{1/p}, \quad a_{n,j}\in\FF.
\]
Hence, the lattice defined by the quasi-norm
\[
g=(x_n)_{n=1}^\infty \mapsto \left\Vert \sum_{n=1}^\infty \lvert x_n\rvert \right\Vert, \quad g\in \XX^\NN,
\]
is $p$-convex. Let $C$ denote its $p$-convexity constant. Let $f_j=(a_{j,n})_{n=1}^\infty\in[0,\infty)^\NN$, $1\le j \le J$. Given $(x_n)_{n=1}^N\in B_\XX^N$ we have
\begin{align*}
\left\Vert \sum_{n=1}^N \left( \sum_{j=1}^J a_{j,n}^p\right)^{1/p} x_n \right\Vert
&\le \left\Vert \sum_{n=1}^N \left( \sum_{j=1}^J (a_{j,n} \, \lvert x_n\rvert)^p\right)^{1/p} \right\Vert\\
&\le C \left( \sum_{j=1}^J \left\Vert \sum_{n=1}^N a_{j,n} \, \lvert x_n\rvert\right\Vert^p\right)^{1/p}\\
&\le C\left( \sum_{j=1}^J \lambda_\XX^p(f_j)\right)^{1/p}.
\end{align*}
Consequently, $\lambda_\XX(\sum_{j=1}^J |f_j|^p)^{1/p}) \le C ( \sum_{j=1}^J \lambda_\XX^p(f_j))^{1/p}$.
\end{proof}

\begin{definition}\label{def:galb}
Let $\XX$ be a quasi-Banach space. We denote $\GG(\XX)=\LL_{\lambda_\XX}$, and we say that $\GG(\XX)$ is the \emph{galb} of $\XX$. The positive cone of $\GG(\XX)$ will be denoted by $\GG^+(\XX)$, and $\GG_b(\XX)$ stands for the closure of $c_{00}$ in $\GG(\XX)$.
\end{definition}

Roughly speaking, it could be said that the galb of a space is a measure of its convexity. The notion of galb was introduced and developed by Turpin, within the more general setting of ``espaces vectoriels \`a convergence'', in a series of papers \cites{Turpin1973a, Turpin1973b} and a monograph \cite{Turpin1976}. In this section we restrict ourselves to galbs of locally bounded spaces and touch only a few aspects of the theory and summarize without proofs the properties that are more relevant to our work.

\begin{proposition}[see \cite{Turpin1976}]
Let $\XX$ be a quasi-Banach space. Then $\GG(\XX)\subseteq\ell_1$, and $\GG(\XX)=\ell_1$ if and only if $\XX$ is locally convex.
\end{proposition}

\begin{proposition}[see \cite{Turpin1976}]\label{prop:GalbGalb}
Let $\XX$ be a quasi-Banach space. Then $\GG(\GG(\XX))=\GG(\XX)$.
\end{proposition}

\begin{proposition}[see \cite{Turpin1976}]
Let $\XX$ be a quasi-Banach space and $0<p\le 1$. Then $\XX$ is $p$-convex if and only if $\ell_p\subseteq \GG(\XX)$.
\end{proposition}

\begin{proposition}[see \cite{Turpin1973a}]\label{prop:bilinear}
Let $\XX$ be a quasi-Banach space. Then the mapping
\[
B\colon \GG(\XX) \times c_0(\XX) \to \XX, \quad \left( (a_n)_{n=1}^\infty, (x_n)_{n=1}^\infty\right) \mapsto \sum_{n=1}^\infty a_n \, x_n
\]
is well-defined, and defines a bounded bilinear map.
\end{proposition}

It is natural to wonder whether the map $B$ defined as in Proposition~\ref{prop:bilinear} can be extended to a continuous bilinear map defined on $\GG(\XX) \times \ell_\infty(\XX)$. In fact, the authors of \cite{KPR1984}, perhaps taking for granted that the answer to this question is positive, defined a sequence $(a_n)_{n=1}^\infty$ to be in the galb of $\XX$ if $\sum_{n=1}^\infty a_n\, x_n$ converges for every bounded sequence $(x_n)_{n=1}^\infty$. If we come to think of it, we obtain the following.

\begin{lemma}\label{lem:KPRG}
Let $\XX$ be a quasi-Banach space and let $f=(a_n)_{n=1}^\infty\in\FF^\NN$. Then, $f\in \GG_b(\XX)$ if and only if $\sum_{n=1}^\infty a_n\, x_n$ converges for every bounded sequence $(x_n)_{n=1}^\infty$ in $\XX$.
\end{lemma}

\begin{proof}
Let $G$ denote the set consisting of all sequences $f=(a_n)_{n=1}^\infty\in\FF^\NN$ such that $\sum_{n=1}^\infty a_n\, x_n$ converges for every bounded sequence $(x_n)_{n=1}^\infty$ in $\XX$. It is routine to check that $G$ is a closed subspace of $\GG(\XX)$ which contains $c_{00}$. Consequently, $\GG_b(\XX)\subseteq G$. Assume that $f=(a_n)_{n=1}^\infty\in \GG(\XX)\setminus \GG_b(\XX)$. Then, by Lemma~\ref{lem:NS}, there are $\delta>0$ and an increasing sequence $(m_k)_{k=1}^\infty$ of non-negative integers such that $\rho ( ( |a_n|)_{n=1+m_{2k-1}}^{m_{2k}} )>\delta$ for all $k\in\NN$.
Consequently, there is $(x_n)_{n=1}^\infty$ in the unit ball of $\ell_\infty(\XX)$ such that
\[
\textstyle
\left\Vert \sum_{n=1+m_{2k-1}}^{m_{2k}} a_n\, x_n\right\Vert\ge \delta, \quad k\in\NN.
\]
We infer that $\sum_{n=1}^\infty a_n\, x_n$ does not converge.
\end{proof}

\begin{corollary}\label{cor:99}Let $\XX$ be a quasi-Banach space. Then the mapping
\[
B'\colon \GG_b(\XX) \times \ell_\infty(\XX) \to \XX, \quad \left( (a_n)_{n=1}^\infty, (x_n)_{n=1}^\infty\right) \mapsto \sum_{n=1}^\infty a_n \, x_n
\]
is well-defined, and defines a continuous bilinear map. Moreover, if $\GG_b(\XX)\subsetneq G\subseteq \GG(\XX)$, then $B'$ can not be extended to a continuous bilinear map defined on $G\times\ell_\infty(\XX)$.
\end{corollary}

\begin{proof}
It follows from Lemma~\ref{lem:KPRG} and, alike the proof of Proposition~\ref{prop:bilinear}, the Open Mapping Theorem.
\end{proof}

In light of Corollary~\ref{cor:99}, the following question arise.

\begin{question}\label{question:A}
Is $\GG(\XX)$ minimal for any quasi-Banach space $\XX$?
\end{question}

Corollary~\ref{cor:96} alerts us of the connection between Question~\ref{question:A} and the existence of lower estimates for $\lambda_\XX$. Lattice concavity also plays a key role when studying galbs of vector-valued spaces.

\begin{definition}
We say that a symmetric function quasi-norm $\lambda$ over $\NN$ \emph{galbs} a quasi-Banach space $\XX$ if $\lambda$ dominates $\lambda_\XX$, i.e., $\LL_\lambda\subseteq \GG(\XX)$. We say that $\lambda$ galbs a function quasi-norm $\rho$ if it galbs $\LL_\rho$. If $\lambda$ galbs itself, we say that $\lambda$ is \emph{self-galbed}.
\end{definition}

\begin{remark}
Given $0<p\le 1$, the function quasi-norm defining $\ell_p$ is self-galbed. More generally, $\lambda_\XX$ is self-galbed for any quasi-Banach space $\XX$ (see Proposition~\ref{prop:GalbGalb}).
\end{remark}

\begin{proposition}\label{prop:vectorgalb}
Let $\rho$ and $\lambda$ be locally absolutely continuous $L$-convex 
function quasi-norms with the Fatou property. 
Suppose that $\lambda$ 
galbs a quasi-Banach space $\XX$. 
If there is $0<p<\infty$ such that $\lambda$ is $p$-concave and $\rho$ is $p$-convex, 
then $\lambda$ galbs $\LL_\rho(\XX)$.
\end{proposition}

\begin{proof}
By Theorem~\ref{thm:MIIpconvex}, 
the pair $(\lambda,\rho)$ has the MII property for some constant $C$. 
Since $\lambda$ galbs $\XX$, there is a constant $K>0$ such that 
$\lambda$ $K$-dominates $\lambda_\XX$. 
Therefore, if $(a_n)_{n=1}^\infty$ is a sequence in $\LL_\lambda$,   
and $f_1,\dotsc,f_N$ belong the unit ball of $\LL_\rho(\XX)$, we have 
\begin{align*}
\rho\left( \left\Vert \sum_{n=1}^N a_n \, f_n \right\Vert\right)
&\le \rho \left(\lambda_\XX \left( (a_n \, \Vert f_n\Vert)_{n=1}^N\right) \right)\\
&\le K \rho \left(\lambda \left( (a_n \, \Vert f_n\Vert)_{n=1}^N\right) \right)\\
&\le CK\lambda \left(\rho \left( (a_n \, \Vert f_n\Vert)_{n=1}^N\right) \right)\\
&\le CK \lambda\left( (a_n)_{n=1}^N \right)\\
&\le CK \lambda((a_n)_{n=1}^\infty).
\end{align*}
Hence $(a_n)_{n=1}^\infty$ belongs the galb of $\LL_\rho(\XX)$.
\end{proof}

Proposition~\ref{prop:vectorgalb} gives, in particular, that if $\lambda$ is a $1$-concave function quasi-norm which galbs $\XX$, then it galbs $L_1(\mu,\XX)$. As we plan to develop an integral for functions belonging to a suitable subspace of $L_1(\mu,\XX)$, the following question arises.

\begin{question}\label{question:B}
Is $\GG(\XX)$ $1$-concave for any quasi-Banach space $\XX$?
\end{question}

Note that a positive answer to Question~\ref{question:B} would yield a positive answer to Question~\ref{question:A}. To properly understand Question~\ref{question:B}, we must go over the state-of-the-art of the theory galbs.

We point out that all known examples suggest a positive answer to Question~\ref{question:B}. Galbs of Lorentz spaces were explored through several papers \cites{SteinWeiss1969,Sjogren1990,Sjogren1992,ColzaniSjogren1999,CCS2007} within the study of convolution operators, and all computed galbs occur to be Orlicz sequence spaces modeled after a concave Orlicz function. Also, Turpin \cite{Turpin1976} proved that the galb of any locally bounded Orlicz space is an Orlicz sequence space modeled after a concave Orlicz function. Recall that an \emph{Orlicz function} is a non-null left-continuous non-decreasing function $\varphi\colon [0,\infty)\to[0,\infty)$ such that $\lim_{t\to 0^+}\varphi(t)=0$. Given an Orlicz function $\varphi$, with the convention that $\varphi(\infty)=\infty$, the gauge
\[
f=(a_n)_{n=1}^\infty\mapsto \lambda_\varphi(f)=\inf\left\{ t>0 \colon \sum_{n=1}^\infty \varphi\left(\frac{a_n}{t} \right)\le 1\right\}, \quad f\in[0,\infty]^\NN
\]
is a function quasi-norm if and only if
\begin{equation}\label{eq:OrliczLC}
\lim_{t\to 0^+}\sup_{u\in (0,1]}\frac{\varphi(tu)}{\varphi(u)}=0.
\end{equation}
(see \cite{Turpin1976}), in which case $\lambda_\varphi$ has the Fatou property. If \eqref{eq:OrliczLC} holds, the Orlicz sequence space $\ell_\varphi$ is the K\"othe space associated with $\lambda_\varphi$.

\begin{proposition}\label{exampleofRMI}
Let $\varphi$ be a concave Orlicz function fulfilling \eqref{eq:OrliczLC}. Then $\lambda_\varphi$ is lattice $1$-concave.
\end{proposition}

\begin{proof}
Let $(f_j)_{j=1}^J$ be a finite family consisting of non-negative sequences. We will prove that
\[
H:=\sum_{j=1}^J \lambda_\varphi(f_j)\le G:=\lambda_\varphi\left(\sum_{j=1}^J f_j\right).
\]
To that end, it suffices to prove that if $G<\infty$ and $0<t<H$, then, $t<G$. Assume without loss of generality that $\lambda_\varphi(f_j)>0$ for all $j$. Then, pick $(t_j)_{j=1}^J$ such that $\sum_{j=1}^J t_j=t$ and $0<t_j<\rho(f_j)$. Then, if $f_j=(a_{j,n})_{n=1}^\infty$, $a_{j,n}<\infty$ for all $n\in\NN$, and
\[
\sum_{n=1}^\infty \varphi\left(\frac{ a_{j,n}}{t_j}\right)>1, \quad j=1,\dots,J.
\]
Consequently,
\[
\sum_{n=1}^\infty \varphi\left(\frac{\sum_{j=1}^J a_{j,n}}{t}\right)
=\sum_{n=1}^\infty \varphi\left(\sum_{j=1}^J \frac{t_j}{t} \frac{a_{j,n}}{t_j}\right)
\ge \sum_{n=1}^\infty \sum_{j=1}^J \frac{t_j}{t} \varphi\left( \frac{ a_{j,n}}{t_j}\right)>1.
\]
Therefore, $t<G$.
\end{proof}

The lattice convexity of spaces of galbs is also quite unknown. It is known that if the gauge $\lambda_\varphi$ associated with an Orlicz function $\varphi$ is  function quasi-norm, so that $\ell_\varphi$ is a quasi-Banach lattice, then there is $p>0$ such that
\begin{equation}\label{eq:OrliczConvex}
\sup_{0<u,t\le 1} \frac{\varphi(t\, u)}{u^p \varphi(t)}<\infty
\end{equation}
(see \cite{Kalton1977}*{Proposition 4.2}). Moreover, if \eqref{eq:OrliczConvex} holds for a given $p$, then  $\ell_\varphi$ is a $p$-convex lattice. Therefore, $\ell_\varphi$ is $L$-convex. The behavior of general spaces of galbs is unknown.
 \begin{question}\label{question:C}
 Is $\lambda_\XX$ an $L$-convex function quasi-norm for any quasi-Banach space $\XX$?
\end{question}
Note that Proposition~\ref{prop:FQG}~\ref{galb:latticecon} partially solves in the positive Question~\ref{question:C}.

\section{Topological tensor products built by means of symmetric function quasi-norms over $\NN$}\label{sec:tensor}
\noindent
\begin{definition}\label{def:tensorlambda}
Let $\XX$ and $\YY$ be quasi-Banach spaces and $\lambda$ be a symmetric minimal function quasi-norm with the Fatou property. We define
\[
\Vert \cdot \Vert_{\XX\otimes_\lambda \YY} \colon \XX\otimes\YY \to [0,\infty)
\]
by
\[
\Vert \tau \Vert_{\XX\otimes_\lambda \YY} =\inf\left\{ \lambda\left( \left( \Vert x_j\Vert \, \Vert y_j\Vert\right)_{j=1}^n\right) \colon \tau=\sum_{j=1}^n x_j\otimes y_j\right\}.
\]
\end{definition}

It is clear that $\Vert \cdot \Vert_{\XX\otimes_\lambda \YY}$ is a semi-quasi-norm whose modulus of concavity is at most that of $\lambda$, and that $\Vert x\otimes y\Vert_{\XX\otimes_\lambda \YY}\le \Vert x\Vert\, \Vert y\Vert$ for all $x\in\XX$ and $y\in\YY$.

\begin{definition}
Let $\XX$ and $\YY$ be quasi-Banach spaces and $\lambda$ be a symmetric minimal function quasi-norm with the Fatou property. The quasi-Banach space built from $\Vert \cdot \Vert_{\XX\otimes_\lambda \YY}$ will be called the \emph{topological tensor product of $\XX$ and $\YY$ by $\lambda$}, and will be denoted by $\XX\otimes_\lambda \YY$. The canonical norm-one bilinear map from $\XX\times\YY$ to $\XX\otimes_\lambda \YY$ given by $(x,y)\mapsto x\otimes y$ will be denoted by $T_\lambda[\XX,\YY]$.
\end{definition}

\begin{proposition}\label{prop:TP}
Let $\XX, \YY, \UU$ and $\VV$ be quasi-Banach spaces, and let $\lambda$ be a symmetric minimal function quasi-norm with the Fatou property. 
\begin{enumerate}[label=(\roman*),leftmargin=*,widest=viii]
\item\label{prop:TP:pC} If $\lambda$ is a function $p$-norm, $0<p\le 1$, then $\XX\otimes_\lambda \YY$ is a $p$-Banach space.
\item\label{prop:TP:Galb} $\GG(\LL_\lambda) \subseteq \GG(\XX\otimes_\lambda \YY)$.
\item\label{prop:TP:UP} If $\lambda$ galbs $\UU$, there is a constant $C$ such that for every bounded bilinear map $B\colon \XX\times \YY\to\UU$ there is a unique linear map $B_\lambda\colon \XX\otimes_\lambda \YY\to\UU$ such that $B_\lambda\circ T_\lambda[\XX,\YY]=B$ and $\Vert B_\lambda\Vert \le C \Vert B\Vert$.
\item\label{prop:TP:CD} If
$R\colon \XX\to\UU$ and $S\colon\YY\to \VV$ are bounded linear operators, then there is a unique bounded linear operator $R\otimes_{\lambda} S\colon \XX\otimes_\lambda\YY\to \UU\otimes_\lambda\VV$ such that $(R\otimes_{\lambda} S)\circ T_\lambda[\XX,\YY]=T_\lambda[\UU,\VV] \circ (R,S)$.
\item\label{prop:TP:DS}
If $\UU$ is complemented in $\XX$ through $R$ and $\VV$ is complemented in $\YY$ through $S$, then $\UU\otimes_\lambda\VV$ is complemented in $\XX\otimes_\lambda\YY$ through $R\otimes_{\lambda} S$. Moreover, if $\UU^c$ and $\VV^c$ are such that $\XX\simeq\UU\oplus \UU^c$ and $\YY\simeq\VV\oplus \VV^c$, then
\[
\XX\otimes_\lambda\YY\simeq (\UU\otimes_\lambda\VV)\oplus (\UU\otimes_\lambda\VV^c)\oplus (\UU^c\otimes_\lambda\VV)\oplus (\UU^c\otimes_\lambda\VV^c).
\]
\item\label{prop:TP:Emb} Let $\rho$ be a symmetric minimal function quasi-norm with the Fatou property. If $\rho$ dominates $\lambda$, then there is a bounded linear map $I\colon \XX\otimes_\rho \YY \to \XX\otimes_\lambda \YY$ such that $I\circ T_\rho[\XX,\YY]= T_\lambda[\XX,\YY]$.
\item\label{prop:TP:Equiv} There is a constant $C$ such that if $(x_j)_{j=1}^\infty$ in $\XX$ and $(y_j)_{j=1}^\infty$ in $\YY$ are such that
\begin{equation}\label{eq:23}
H=\lambda\left(\left(\Vert x_j\Vert \,\Vert y_j\Vert\right)_{j=1}^\infty\right)<\infty.
\end{equation}
then $\sum_{j=1}^\infty x_j\otimes y_j$ converges in $\XX\otimes_\lambda \YY$ to a vector $\tau\in \XX\otimes_\lambda \YY$ with $\Vert \tau \Vert_{\XX\otimes_\lambda \YY}\le C H$. Conversely, for all $\tau \in \XX\otimes_\lambda \YY$ and $\varepsilon>0$ there are $(x_n)_{n=1}^\infty$ in $\XX$ and $(y_n)_{n=1}^\infty$ in $\YY$ such that, if
\[
f:=(\Vert x_j\Vert \,\Vert y_j\Vert)_{j=1}^\infty,
\]
then $\lambda(f)\le\varepsilon+ C \Vert \tau\Vert_{\XX\otimes_\lambda \YY}$ and $\tau=\sum_{j=1}^\infty x_j\otimes y_j$. Moreover, if $\lambda$ is a function $p$-norm, we can pick $C=1$. And, if $\XX_0$ and $\YY_0$ are dense subspaces of $\XX$ and $\YY$ respectively, we can pick $x_j\in\XX_0$ and $y_j\in\YY_0$ for all $j\in\NN$.
\item\label{prop:TP:Y=Fn} If $\lambda$ galbs $\XX$ and $\YY$ is finite dimensional, then $\XX\otimes_\lambda \YY\simeq\XX^n$, where $n=\dim(\YY)$. To be precise, if $(\yy_j)_{j=1}^n$ is a basis of $\YY$, the map $R\colon\XX^n\to \XX\otimes_\lambda \YY$ given by $(x_j)_{j=1}^n\mapsto \sum_{j=1}^n x_j\otimes \yy_j$ is an isomorphism.
\item\label{prop:TP:SP} If $\lambda$ galbs $\XX$ and $\YY$ has the point separation property, then $\Vert \cdot \Vert_{\XX\otimes_\lambda \YY}$ is a quasi-norm on $\XX\otimes \YY$.
\end{enumerate}
\end{proposition}

\begin{proof}
A simple computation yields \ref{prop:TP:pC}.

Let $f=(a_k)_{k=1}^\infty\in[0,\infty)^\NN$, and let $(\tau_k)_{k=1}^m$ in $\XX\otimes \YY$ be such that $\Vert \tau_k\Vert_{\XX\otimes_\lambda \YY}\le 1$. Then, given $\varepsilon>0$, for each $k=1,\dots,m$ there is an expansion
\[
\tau_k=\sum_{j=1}^{n_k} b_{k,j} \, x_{k,j} \otimes y_{k,j},
\]
with $\max\{\Vert x_{k,j}\Vert , \Vert y_{k,j}\Vert \}\le 1$ for all $(k,j)\in\Nt:=\{(k,j) \in\NN^2 \colon 1\le k \le m, \ 1\le j\le n_k\}$ and $\lambda((b_{k,j})_{k=1}^{n_j})\le 1 +\varepsilon$. The expansion
\[
\tau:=\sum_{k=1}^m a_k\, \tau_k=\sum_{(k,j)\in\Nt} a_k \, b_{k,j} \, x_{k,j}\otimes y_{k,j}
\]
gives
\[
\Vert \tau\Vert_{\XX\otimes_\lambda \YY}\le \Vert (a_k \, b_{k,j})_{(k,j)\in\Nt}\Vert_\lambda \le (1+\varepsilon) \lambda_{\LL_\lambda}(f).
\]
Consequently, $\lambda_{\XX\otimes_\lambda\YY}(f) \le \lambda_{\LL_\lambda}(f)$, and we obtain \ref{prop:TP:Galb}. 

Let us prove \ref{prop:TP:UP}. Let $C$ be such that $\Vert \sum_{j=1}^n a_j \, u_j\Vert \le C \lambda( (a_j)_{j=1}^n)$ for all $(a_j)_{j=1}^n$ in $[0,\infty)^n$ and $(u_j)_{j=1}^n$ in $B_\UU$. Given a bounded bilinear map $B\colon \XX\times\YY\to\UU$, let $B_0\colon\XX\otimes\YY\to\UU$ be the linear map defined by $B(x\otimes y)=B(x,y)$. Given $\tau=\sum_{j=1}^n x_k\otimes y_k\in \XX\otimes\YY$ we have
\[
\Vert B_0(\tau) \Vert
\le C \lambda( ( \Vert B(x_j,y_j)\Vert)_{j=1}^n)
\le C\Vert B\Vert \lambda( ( \Vert x_j\Vert \, \Vert y_j \Vert)_{j=1}^n).
\]
Consequently, $\Vert B_0(\tau) \Vert \le C \Vert B\Vert \Vert \tau\Vert_{\XX\otimes_\lambda\YY}$. We infer that $B_0$ `extends' to an operator as desired.

Now we prove \ref{prop:TP:CD}. Let $\tau\in\XX\otimes\YY$. The mere definitions of the semi-quasi-norms involved give
\begin{align*}
\Vert (R\otimes_\lambda S)\tau\Vert_{\UU\otimes_\lambda\VV}
&\le\inf\Big\{\lambda\left( \left( \Vert R(x_j)\Vert \, \Vert S(y_j)\Vert\right)_{j=1}^n\right)\colon \tau=\sum_{j=1}^n x_j\otimes y_j\Big\}\\
&\le\Vert R\Vert \, \Vert S\Vert\, \Vert \tau\Vert_{\XX\otimes_\lambda\YY}.
\end{align*}

For statement \ref{prop:TP:DS}, it suffices to consider the case when $\VV=\YY$ and $S_v=\Id_\YY$. Let $I\colon \UU\to \XX$ and $P\colon\XX\to\UU$ be such that $P\circ I=\Id_\UU$. Then $(P\otimes_{\lambda}\Id_\YY) \circ (I\otimes_{\lambda}\Id_\YY)=\Id_{\UU\otimes_{\lambda}\YY}$. Let $J\colon \UU^c\to \XX$ and $Q\colon\XX\to\UU^c$ be such that $Q\circ J=\Id_{\UU^c}$ and $J\circ Q+I\circ P=\Id_\XX$. Then
\[
(I\otimes_{\lambda}\Id_\YY) \circ (P\otimes_{\lambda}\Id_\YY) + (J\otimes_{\lambda}\Id_\YY) \circ (Q\otimes_{\lambda}\Id_\YY)=\Id_{\XX\otimes_{\lambda}\YY}.
\]

Statement \ref{prop:TP:Emb} is immediate from definition.

Let us prove \ref{prop:TP:Equiv}. Assume without lost of generality that $\lambda$ is function $p$-norm for some $0<p\le 1$. If \eqref{eq:23} holds, then $\sum_{j=1}^\infty x_j\otimes y_j$ is a Cauchy series. Therefore, it converges to $\tau\in\XX\otimes_\lambda\YY$. The continuity of the quasi-norm $\Vert \cdot\Vert_{\XX\otimes_\lambda\YY}$ yields
\[
\Vert \tau\Vert_{\XX\otimes_\lambda\YY}=\lim_m \left\Vert \sum_{j=1}^m x_j\otimes y_j\right\Vert_{\XX\otimes_\lambda\YY} \le H.
\]
Conversely, let $\tau\in\XX\otimes_\lambda\YY$ and $\varepsilon>0$. 
Assume that $\XX_0$ and $\YY_0$ are dense subspaces of $\XX$ and $\YY$ respectively. 
Pick $(\tau_n)_{n=1}^\infty$ in $\XX_0\otimes\YY_0$ such that $\lim_n \Vert \tau-\tau_n\Vert_{\XX\otimes_\lambda\YY}=0$, and pick a sequence $(\varepsilon_n)_{n=1}^\infty$ of positive numbers with
\[
\varepsilon_1>\Vert \tau\Vert_{\XX\otimes_\lambda\YY}> \left(\sum_{n=1}^\infty \varepsilon_n^p\right)^{1/p}-\varepsilon.
\]
Passing to a subsequence we can suppose that $\Vert \tau_n-\tau_{n-1} \Vert_{\XX\otimes_\lambda\YY}< \varepsilon_n$ for all $n\in\NN$, with the convention $\tau_0=\tau$. Therefore, for all $n\in\NN$, we can write
\[
\tau_n-\tau_{n-1}=\sum_{j=1}^{j_n} x_{j,n} \otimes y_{j,n}, \ \ R_n:=\lambda \left( \left( \Vert x_{j,n} \Vert \, \Vert y_{j,n}\Vert \right)_{j=1}^{j_n}\right)<\varepsilon_n.
\]
Let $\Nt=\{(j,n)\in\NN^2 \colon 1\le j \le j_n \}$. Then 
\[
\lambda\left( \left( \Vert x_{j,n} \Vert \, \Vert y_{j,n} \Vert \right)_{(j,n)\in\Nt}\right) \le\left( \sum_{n=1}^\infty R_n^p\right)^{1/p} \le \varepsilon+\Vert \tau\Vert_{\XX\otimes_\lambda\YY}.
\]
Hence, we can safely define $\tau'= \sum_{(j,n)\in\Nt} x_{j,n} \otimes y_{j,n}$, and we have
\[
\tau'=\sum_{n=1}^\infty \sum_{j=1}^{j_n} x_{j,n} \otimes y_{j,n}=\sum_{n=1}^\infty (\tau_n-\tau_{n-1})=\lim_n\tau_n=\tau.
\]

Now we prove \ref{prop:TP:Y=Fn}. The mapping $R$ is linear and bounded, and $R(\XX^n)$ spans $\XX\otimes_\lambda\YY$. Since $\lambda$ galbs $\XX$, there is a bounded linear map $S\colon\XX\otimes_\lambda\YY\to \XX^n$ such that $S(x\otimes \yy_j)=x\,\ee_j$ for all $x\in\XX$ and $j=1,\dotsc,n$. Taking into account that $S\circ R=\Id_{\XX^n}$, we are done.

Finally, let $\VV$ be finite-dimensional subspace of $\YY$. Since $\VV$ is complemented in $\YY$, $\XX\otimes_\lambda \VV$ is complemented in
$\XX\otimes_\lambda \YY$ via the canonical map. Hence, it suffices to consider the case when $\YY$ is finite dimensional. In this particular case, statement \ref{prop:TP:SP} follows from \ref{prop:TP:Y=Fn}.
\end{proof}

\section{Topological tensor products as spaces of functions and integrals for spaces of vector-valued functions}\label{sec:integration}
\noindent
Let us give another approach to the proof of Proposition~\ref{prop:TP}~\ref{prop:TP:SP}. Given quasi-Banach spaces $\XX$ and $\YY$, let $B\colon \XX \times \YY\to \ell_\infty(B_{\YY^*},\XX)$ be defined by $B(x,y)(y^*)=y^*(y) x$. Since $B$ is linear and bounded, if $\lambda$ galbs $\XX$, there is a bounded linear map $B_\lambda\colon \XX \otimes_\lambda \YY\to \ell_\infty(B_{\YY^*},\XX)$ given by $B_\lambda(x\otimes y)(y^*)=y^*(y) x$. If $\YY$ has the point separation property, then $B_\lambda$ is one-to-one on $\XX \otimes \YY$. Consequently, no vector in $\XX \otimes \YY$ is norm-zero. Note the injectivity of $B_\lambda$ on $\XX \otimes \YY$ does not implies the injectivity of $B_\lambda$ on its closure $\XX \otimes_\lambda \YY$. That is, we can not, a priori, identify vectors in $\XX \otimes_\lambda \YY$ with functions defined over $B_{\YY^*}$. More generally, if $\YY$ embeds in $\FF^{\Omega}$ for some set $\Omega$, then $\XX\otimes\YY$ embeds into $\XX^\Omega$, and it is natural to wonder if the character of the members of $\XX\otimes\YY$ is preserved when taking the completions, that is, if we can regard the vectors in $\XX\otimes_\lambda\YY$ as $\XX$-valued functions defined on $\Omega$. In this section, we address this question in the case when $\YY$ is a K\"othe space. 

Given a quasi-Banach space $\XX$ and a $\sigma$-finite measure space $(\Omega,\Sigma,\mu)$ we have a canonical linear map
\[
J[\XX,\mu]\colon \XX\otimes L_0(\mu)\to L_0(\mu,\XX), \quad x\otimes f \mapsto x f.
\]
It is routine to check that $J[\XX,\mu]$ is one-to-one. Suppose that $\lambda$ is a symmetric function quasi-norm and $\rho$ is a function quasi-norm over $(\Omega,\Sigma,\mu)$ such that $\lambda$ is $p$-concave and $\rho$ is $p$-convex for some $0<p<\infty$. Then $\lambda$ is minimal (see Corollary~\ref{cor:96}). So, we can safely define $\XX\otimes_\lambda \LL_\rho$. If, moreover, $\lambda$ galbs $\XX$, then $\lambda$ also galbs $\LL_\rho(\XX)$ (see Proposition~\ref{prop:vectorgalb}). 
Hence, if $\rho$ has the weak Fatou property, there is a bounded linear canonical map
\[
J[\rho,\XX,\lambda]\colon \XX\otimes_\lambda \LL_\rho\to \LL_\rho(\XX), \quad x\otimes f \mapsto x f.
\]
Consider the range
\[
\LL_\rho^\lambda(\XX):=J[\rho,\XX,\lambda] (\XX\otimes_\lambda \LL_\rho)
\]
of this operator endowed with the quotient topology. If $J[\rho,\XX,\lambda]$ is one-to-one, then $\LL_\rho^\lambda(\XX)$ is a space isometric to $\XX\otimes_\lambda \LL_\rho$ which embeds continuously into $\LL_\rho(\XX)$. This is our motivation to studying the injectivity of $J[\rho,\XX,\lambda]$. Vogt \cite{Vogt1967} gave a positive answer to this question in the case when $\lambda$ is the function quasi-norm associated with $\ell_p$ for some $0<p\le 1$ and $\rho$ is the function quasi-norm associated with $L_q(\mu)$ for some $p\le q\le \infty$. A detailed analysis of the proof of \cite{Vogt1967}*{Satz 4} reveals that it depends heavily on the fact that $\lambda$ is both $p$-convex and $p$-concave and $\rho$ is both $q$-convex and $q$-concave. So, it is hopeless to try to extend this result using analogous ideas. In this paper, we use an approach based on conditional expectations.

Before going on, let us mention that if $\lambda$ is the function quasi-norm associated with $\ell_1$ (and $\rho$ and $\XX$ are $1$-convex), then a routine computation yields that $J[\rho,\XX,\lambda]$ is an isometric embedding when restricted to $\XX\otimes \Sp(\mu)$. We infer that $J[\rho,\XX,\lambda]$ is an isometric embedding and that $\LL_\rho^\lambda(\XX)$ consists of all strongly measurable functions in $\LL_\rho(\XX)$.

\begin{lemma}\label{lem:TO}
Let $\lambda$ be a minimal symmetric function quasi-norm. For $i=1,2$, let $\rho_i$ be a function quasi-norm with the weak Fatou property over a $\sigma$-finite measure space $(\Omega_i,\Sigma_i,\mu_i)$, and let $\XX_i$ be a quasi-Banach space galbed by $\lambda$. Suppose that the bounded linear operators $S\colon\XX_1\to\XX_2$, $T\colon \LL_{\rho_1} \to \LL_{\rho_2}$ and $R\colon \LL_{\rho_1}(\XX_1)\to \LL_{\rho_2}(\XX_2)$ satisfy
\[
R(x\, f)=S(x) \,T(f), \quad x\in\XX_1,\ f\in \LL_{\rho_1}.
\]
Then, $R$ restricts to a bounded linear map from $\LL_{\rho_1}^{\lambda}(\XX_1)\to \LL_{\rho_2}^\lambda(\XX_2)$.
\end{lemma}
\begin{proof}Our assumptions yield a commutative diagram
\[\xymatrix{
\XX_1 \otimes_\lambda \LL_{\rho_1} \ar[d]_{J[\rho_1,\XX_1,\lambda]} \ar[rr]^{ S \otimes_\lambda T} && \XX_2 \otimes_\lambda \LL_{\rho_2} \ar[d]^{J[\rho_2,\XX_2,\lambda]} \\
\LL_{\rho_1}(\XX_1) \ar[rr]_{R} && \LL_{\rho_2}(\XX_2).
}\]
We infer that $R$ maps the range of the map $J[\rho_1,\XX_1,\lambda]$ into the range of the map $J[\rho_2,\XX_2,\lambda]$. That is, there is a linear map $R[\lambda]\colon \LL_{\rho_1}(\XX_1)\to \LL_{\rho_2}(\XX_2)$ such that the diagram
\[\xymatrix{
\XX_1 \otimes_\lambda \LL_{\rho_1} \ar[d]_{J[\rho_1,\XX_1,\lambda]} \ar[rr]^{ S \otimes_\lambda T} && \XX_2 \otimes_\lambda \LL_{\rho_2} \ar[d]^{J[\rho_2,\XX_2,\lambda]} \\
\LL_{\rho_1}^\lambda(\XX_1) \ar[rr]_{R[\lambda]} && \LL_{\rho_2}^\lambda(\XX_2)
}\]
commutes. Since both $\LL_{\rho_1}^\lambda(\XX_1)$ and $\LL_{\rho_2}^\lambda(\XX_2)$ are endowed with the quotient topology and $S \otimes_\lambda T$ is continuous, so is $R[\lambda]$.
\end{proof}

Let $\lambda$ be a $1$-concave symmetric function quasi-norm that galbs a quasi-Banach space $\XX$. Let $(\Omega,\Sigma,\mu)$ be a $\sigma$-finite measure space. If $\rho$ is the function quasi-norm defining $L_1(\mu)$, we denote $L_1^\lambda(\mu,\XX)=\LL_\rho^\lambda(\XX)$. Given $A\in\Sigma$, we set $L_1^\lambda(A,\mu,\XX)=L_1^\lambda(\mu|_A,\XX)$. The bounded linear operator
\[
I[\mu] \colon L_1(\mu)\to \FF, \quad f\mapsto \int_\Omega f\, d\mu
\]
yields a bounded linear operator
\[
I[\mu,\XX,\lambda]\colon \XX\otimes_\lambda L_1(\mu)\to \XX, \quad x\otimes f\mapsto x\int_\Omega f \, d\mu.
\]
\begin{definition}
Suppose that a $1$-concave symmetric function quasi-norm $\lambda$ galbs a quasi-Banach space $\XX$. We say that the pair $(\lambda,\XX)$ is \emph{amenable} if $I[\mu,\XX,\lambda](\tau)=0$ whenever $(\Omega,\Sigma,\mu)$ is a $\sigma$-finite measure and $\tau\in \XX\otimes_\lambda L_1(\mu)$ satisfies $J[L_1(\mu),\XX,\lambda](\tau)=0$.
\end{definition}
In other words, $(\lambda,\XX)$ is amenable if and only if for every $\sigma$-finite measure $\mu$ there is an operator
\[
\II[\mu,\XX,\lambda]\colon L_1^\lambda(\mu,\XX)\to \XX
\]
such that the diagram
\[\xymatrix{
\XX \otimes_\lambda \LL_1(\mu) \ar[d]_{J[L_1(\mu),\XX,\lambda]} \ar[drr]^{I[\mu,\XX,\lambda]}& &\\
L_1^\lambda(\mu,\XX) \ar[rr]_{\II[\mu,\XX,\lambda]}& &\XX
}\]
commutes. The bounded linear operator $\II[\mu,\XX,\lambda]$ satisfies
\[
\II[\mu,\XX,\lambda] (x\, f)= x\int_\Omega f\, d\mu, \quad x\in\XX, \ f\in L_1(\mu).
\]
So, we must regard it as `integral' for functions in $L_1^\lambda(\mu,\XX)$. Loosely speaking, that $(\lambda,\XX)$ is amenable means that there is an integral for functions in $L_1^\lambda(\mu,\XX)$.
\begin{definition}
Let $\XX$ be a quasi-Banach space. We say that a net $(T_i)_{i\in I}$ in $\LT(\XX)$ is a \emph{bounded approximation of the identity} if $\sup_{i} \Vert T_i\Vert <\infty$ and $\lim_i T_i(x)=x$ for all $x\in\XX$. We say that $\XX$ has the \emph{BAP} if it has a bounded approximation of the identity consisting of finite-rank operators.
\end{definition}

Note that if a net $(T_i)_{i\in I}$ in $\LT(\XX)$ is uniformly bounded then the set $\{x \in\XX \colon \lim_i T_i(x)=x\}$ is closed. This yields the following elementary result.
\begin{lemma}\label{lem:PiP}
Let $\XX$ be a quasi-Banach space. Let $(P_i)_{i\in I}$ be a net consisting of uniformly bounded projections with $P_j\circ P_i=P_i$ if $i\le j$ and $\cup_{i\in I}P_i(\XX)$ is dense in $\XX$. Then $(P_i)_{i\in I}$ is a bounded approximation of the identity.
\end{lemma}

If $\rho$ satisfies \ref{FQN:LC}, then for every $A\in\Sigma(\mu)$ we have a bounded linear map
\[
S[A,\rho]\colon L_\rho\to L_1(A,\mu), \quad f\mapsto f|_A.
\]
\begin{theorem}\label{thm:one-to-one}
Let $\lambda$ be a $1$-concave symmetric function quasi-norm, let $\rho$ be a leveling function quasi-norm with the weak Fatou property over a $\sigma$-finite measure space $(\Omega,\Sigma,\mu)$, and let $\XX$ be a quasi-Banach space. Suppose that $(\lambda,\XX)$ is amenable. Then $J[\rho,\XX,\lambda]$ is one-to-one.
\end{theorem}

\begin{proof}
Let $A\in\Sigma(\mu)$. By Lemma~\ref{lem:LevInt}, $\rho$ satisfies \ref{FQN:LC}. Therefore, for each quasi-Banach space $\YY$ there is a bounded linear operator
\[
S[A,\rho,\YY]\colon \LL_\rho(\YY)\to L_1(A,\mu,\YY), \quad f\mapsto f|_A.
\]
Set $S[A,\rho,\FF]=S[A,\rho]$. It is routine to check that the diagram
\[\xymatrix{
\XX\otimes_\lambda \LL_\rho \ar[rr]^-{\Id_\XX\otimes_\lambda S[A,\rho]} \ar[d]_{J[\rho,\XX,\lambda]} &&\XX\otimes_\lambda \LL_1(A,\mu) \ar[d]^{J[L_1(A,\mu),\XX,\lambda]} \\
\LL_\rho(\XX) \ar[rr]_-{S[A,\rho,\XX]} &&L_1(A,\mu,\XX)
}\]
commutes. Using that $(\lambda,\XX)$ is amenable we obtain the commutative diagram
\begin{equation}\label{key:diagram}
\xymatrix{
\XX\otimes_\lambda \LL_\rho \ar[rr]^-{\Id_\XX\otimes_\lambda S[A,\rho]} \ar[d]_{J[\rho,\XX,\lambda]} &&\XX\otimes_\lambda \LL_1(A,\mu) \ar[d]_{J[L_1(A,\mu),\XX,\lambda]} \ar[drr]^-{I[\mu|_A,\XX,\lambda]}&&\\
\LL_\rho^\lambda(\XX) \ar[rr]_-{S[A,\rho,\XX]} &&L_1^\lambda(A,\mu,\XX) \ar[rr]_-{\II[\mu|_A,\XX,\lambda]}&&\XX
}
\end{equation}

Suppose that $\mu(\Omega)<\infty$. Let $\Sigma_0$ be a finite sub-$\sigma$-algebra. If $\Sigma_0$ is generated by the partition $(A_j)_{j=1}^n$ of $\Omega$ consisting of nonzero measure sets, then  
\[
\EE(\rho,\Sigma_0)= \sum_{j=1}^n \frac{\chi_{A_j}}{\mu(A_j)} I[\mu|_{A_j}] \circ S[A_j,\rho].
\]
By Proposition~\ref{prop:TP}~\ref{prop:TP:Y=Fn}, there is an isomorphism $S\colon \XX^n \to \XX\otimes_\lambda \LL_\rho(\Sigma_0)$ such that
\[
S((x_j)_{j=1}^n) =\sum_{j=1}^n x_j \otimes \frac{\chi_{A_j}}{\mu(A_j)}, \quad x_j\in\XX.
\]
Therefore,
\[
\Id_\XX\otimes_\lambda \EE(\rho,\Sigma_0)= S\circ (I[\mu|_{A_j},\XX,\lambda] \circ (\Id_\XX\otimes S[A_j,\rho]))_{j=1}^n.
\]
Combining this identity with the commutative diagrams \eqref{key:diagram} associated with each set $A_j$ yields a bounded linear map $R\colon \LL_\rho^\lambda(\XX)\to \XX \otimes_\lambda \LL_\rho(\Sigma_0)$ such that the diagram
\[\xymatrix{
\XX \otimes_\lambda \LL_\rho \ar[d]_{J[\rho,\XX,\lambda]} \ar[rrd]^-{\Id_\XX\otimes_\lambda \EE(\rho,\Sigma_0)} && \\
\LL_\rho^\lambda(\XX) \ar[rr]_-{R} && \XX \otimes_\lambda \LL_\rho(\Sigma_0) \\
}\]
commutes. The operators $\Id_\XX\otimes_\lambda \EE(\rho,\Sigma_0)$ are uniformly bounded projections. Let $(\Sigma_i)_{i\in I}$ a non-decreasing net of finite $\sigma$-algebras whose union generates $\Sigma$. By Lemma~\ref{lem:PiP}, $(\Id_\XX\otimes_\lambda \EE(\rho,\Sigma_i))_{i\in I}$ is a bounded approximation of the identity. We infer that $J[\rho,\XX,\lambda]$ is one-to-one, as wanted, in the particular case that $\mu(\Omega)<\infty$. 

In general, let $R[A,\XX]\colon \LL_\rho(\XX) \to \LL_\rho(A,\XX)$ be the canonical projection on a set $A\in\Sigma(\mu)$. Set $R[A]=R[A,\FF]$. Since $R[A,\XX]$ is bounded, applying Lemma~\ref{lem:TO} yields a bounded linear operator $R[A,\XX,\lambda]$ such that the diagram
\[\xymatrix{
\XX \otimes_\lambda \LL_\rho \ar[d]_{J[\rho,\XX,\lambda]} \ar[rr]^{\Id_\XX\otimes_\lambda R[A]} && \XX \otimes_\lambda \LL_\rho(A) \ar[d]^{J[\rho|_A,\XX,\lambda]} \\
\LL_\rho^\lambda(\XX) \ar[rr]_{R[A,\XX,\lambda]} && \LL_\rho^\lambda(A,\XX)
}\]
commutes. Let $(A_n)_{n=1}^\infty$ be a non-decreasing sequence in $\Sigma(\mu)$ whose union is $\Omega$. By Lemma~\ref{lem:PiP}, $(\Id_\XX\otimes_\lambda R[A_n])_{n=1}^\infty$ is a bounded approximation of the identity. Since $J[\rho|_{A_n},\XX,\lambda]$ is one-to-one (by the previous particular case), it follows that $J[\rho,\XX,\lambda]$ is one-to-one.
\end{proof}

Notice that the applicability of Theorem~\ref{thm:one-to-one} depends on the existence of amenable pairs. In the optimal situation, we would be able to choose $\lambda$ to be the smallest symmetric function quasi-norm which galbs the quasi-Banach space $\XX$. Thus, the following question arises.

\begin{question}\label{qt:ame}
Let $\XX$ be a quasi-Banach space. Is $(\lambda_\XX,\XX)$ amenable?
\end{question}

As long as there is no general answer to Question~\ref{qt:ame}, we will focus on the spaces of galbs that have appeared in the literature. We next prove that for all of them Question~\ref{qt:ame} has a positive answer.

\begin{theorem}\label{thm:Orlicz}
Let $\varphi$ be a concave Orlicz function fulfilling \eqref{eq:OrliczLC}. Suppose that $\lambda_\varphi$ galbs a quasi-Banach space $\XX$. Then $(\lambda_\varphi,\XX)$ is amenable.
\end{theorem}

\begin{proof}
Assume that $\varphi(1)=1$. Assume by contradiction that there is a $\sigma$-finite measure space $(\Omega,\Sigma,\mu)$, a positive sequence $\alpha=(a_j)_{j=1}^\infty$ in $\ell_\varphi$, a sequence $(f_j)_{j=1}^\infty$ in the unit ball of $L_1(\mu)$, and a sequence $(x_j)_{j=1}^\infty$ in the unit ball of $\XX$ such that $\sum_{j=1}^\infty a_j \, x_j\, f_j=0$ in $\LL_\varphi(\XX)$ and
\[
x:=\sum_{j=1}^\infty a_j \, x_j \int_\Omega f_j\, d\mu\not=0.
\]
The following claim will be used a couple of times.

\noindent\textbf{Claim.} If $(\Omega_k)_{k=1}^\infty$ is a non-decreasing sequence in $\Sigma(\mu)$ such that $\Omega\setminus\cup_{k=1}^\infty \Omega_k$ is a null set, then $\sum_{j=1}^\infty a_j \, x_j \int_{\Omega_k} f_j\, d\mu\not=0$ for some $k\in\NN$. 

\noindent\emph{Proof of the claim.} Since $\lim_k \int_{\Omega_k} f_j\, d\mu=\int_\Omega f_j\, d\mu$ for all $j\in\NN$ and $\lambda_\varphi$ is dominating, we have
\[
\lim_k \left\Vert \left(a_j \int_\Omega f_j\,d\mu\right)_{j=1}^\infty-\left(a_j \int_{\Omega_k} f_j\,d\mu\right)_{j=1}^\infty\right\Vert_\varphi=0.
\]
Since $\ell_\varphi$ embeds continuously in $\GG_b(\XX)$,
\[
\lim_k \left\Vert\sum_{j=1}^\infty a_j \, x_j \int_{\Omega} f_j\,d\mu- \sum_{j=1}^\infty a_j \, x_j \int_{\Omega_k} f_j \,d\mu \right\Vert=0.
\]
This limit readily gives our claim.

The claim allow us assume that $\mu(\Omega)<\infty$. By Proposition~\ref{prop:TP}~\ref{prop:TP:Equiv}, we can assume that $f_j\in\Sp(\mu)$ for all $j\in\NN$. Also, we can assume without lost of generality that $\lambda_\varphi(\alpha)<1$, so that $\sum_{j=1}^\infty \varphi(a_j)<1$. Set
\[
F_m=\sum_{j=m+1}^\infty \varphi(a_j) \lvert f_j\rvert, \quad m\in\NN\cup\{0\}.
\]
We have $\int_\Omega F_0\, d\mu<\infty$. Therefore, $F_0<\infty$ a.e. By Severini–Egorov theorem, $\lim_m F_m=0$ quasi-uniformly. By Proposition~\ref{prop21}, there is an increasing sequence $(J_n)_{n=1}^\infty$ such that, if
\[
G_n=\sum_{j=1}^{J_n} a_j \, x_j\, f_j, \quad n\in\NN,
\]
then $\lim_n G_n=0$ a.e. Taking into account the claim, we can assume without lost of generality that $\lim_m F_m=0$ uniformly and that $\lim_n G_n=0$ pointwise.

Pick $0<\varepsilon<1$. 
There is $m_0\in\NN$ such that $\lambda_\varphi ((a_j)_{m_0+1}^\infty )<\varepsilon$, i.e.,
\[
A:=\sum_{j=m_0+1}^\infty \varphi\left( \frac{a_j}{\varepsilon} \right)<1.
\]
Let $m\ge m_0$ be such that
\[
F_m(\omega)\le \frac{\varepsilon(1-A)}{\mu(\Omega)}, \quad \omega\in\Omega.
\]
Let $\Sigma_0$ be a finite $\sigma$-algebra such that $f_j$ is $\Sigma_0$-measurable for all $1\le j \le m$. Let $(A_h)_{h=1}^H$ be a partition of $\Omega$ which generates $\Sigma_0$. Pick points $\omega_h\in A_h$ for each $1\le h \le H$, and set
\[
g_j=f_j-\sum_{h=1}^H f_j(\omega_h) \chi_{A_h},
\quad j\in\NN.
\]
Since $g_j=0$ for all $1\le j \le m$ we have
\begin{align*}
x&=\lim_n\sum_{j=1}^{J_n} a_j x_j \int_\Omega f_j\, d\mu -\sum_{h=1}^H \mu(A_h)\lim_n G_n(\omega_h)\\
&=\lim_n\sum_{j=1}^{J_n} a_j x_j \int_\Omega g_j \, d\mu
=\lim_n\sum_{j=m+1}^{J_n} a_j x_j \int_\Omega g_j \, d\mu.
\end{align*}
Notice that
\[
\left\lVert 
\sum_{j=m+1}^{J_n} a_j x_j \int_\Omega g_j \, d\mu
\right\rVert
\le \lambda_\XX((a_jb_j)_{j=m+1}^{\infty}), 
\]
where $b_j=\lvert \int_\Omega g_j \, d\mu\rvert$. 
Recall that if a sequence $(u_n)_{n=1}^{\infty}$ converges to $x$ in $\XX$, 
then $\lVert x\rVert\le\kappa\liminf\lVert u_n\rVert$, 
where $\kappa$ is the modulus of concavity of $\XX$.
Therefore, since $\lambda_\varphi$ galbs $\XX$, we have 
\[
\lVert x\rVert\le\kappa\lambda_\XX((a_jb_j)_{j=m+1}^{\infty}) 
\le\kappa C\lambda_\varphi((a_jb_j)_{j=m+1}^{\infty}), 
\]
for some constant $C>0$. 
Now let us see that $\lambda_\varphi ((a_j b_j)_{j=m+1}^\infty)\le\varepsilon$. 
Using the concavity of $\varphi$ and that $\varepsilon<1$, we have
\begin{align*}
\sum_{j=m+1}^\infty \varphi\left(\frac{a_jb_j}{\varepsilon}\right)
&\le \sum_{j=m+1}^\infty \max\{1,b_j\} \varphi\left(\frac{a_j}{\varepsilon}\right)\\
&\le \sum_{j=m+1}^\infty \left( 1 +\sum_{h=1}^H |f_j(\omega_h)| \mu(A_h) \right) \varphi\left(\frac{a_j}{\varepsilon}\right)\\
&\le \sum_{j=m+1}^\infty \varphi\left(\frac{a_j}{\varepsilon}\right) +\sum_{h=1}^H \sum_{j=m+1}^\infty \frac{1}{\varepsilon} \mu(A_h) |f_j(\omega_h)| \varphi(a_j)\\
&= \sum_{j=m+1}^\infty \varphi\left(\frac{a_j}{\varepsilon}\right) +\frac{1}{\varepsilon}\sum_{h=1}^H \mu(A_h) F_m(\omega_h)\\
&\le A +\frac{1}{\varepsilon}\sum_{h=1}^H \mu(A_h) \frac{\varepsilon(1-A)}{\mu(\Omega)}=1.
\end{align*}
Therefore $\lVert x\rVert\le kC\varepsilon$. 
Letting $\varepsilon$ tend to $0$ we arise to absurdity.
\end{proof}

Given a quasi-Banach space $\XX$, a $\sigma$-finite measure space $(\Omega,\Sigma,\mu)$, a symmetric function quasi-norm $\lambda$ such that $(\lambda,\XX)$ is amenable, and a function $f\colon\Omega\to\XX$, we say that $f$ is $\lambda$-\emph{integrable} if $f\in L_1^{\lambda}(\mu,\XX)$, and we write
\[
\int_\Omega^\lambda f\, d\mu =\II[\mu,\XX,\lambda](f).
\]
A natural question is whether $\int_\Omega^\lambda f\, d\mu =\II[\mu,\XX,\lambda](f)$ really depends on $\lambda$. That is, do we have $\II[\mu,\XX,\lambda_1](f)=\II[\mu,\XX,\lambda_2](f)$ whenever
$(\lambda_1,\XX)$ and $(\lambda_2,\XX)$ are amenable pairs? This question is equivalent to the following one. Given function quasi-norms $\rho_1$ and $\rho_2$ over the same $\sigma$-finite measure space $(\Omega,\Sigma,\mu)$ we define a function quasi-norm $\rho_1\cap\rho_2$ by
\[
(\rho_1\cap\rho_2)(f)=\inf \{ \rho_1(g)+\rho_2(h) \colon g,h\in L_0^+(\mu), \ f=g+h\},
\]
for each $f\in L_0^+(\mu)$. 
It can be proved that if $\rho_1$ and $\rho_2$ are $p$-concave (resp.\ $p$-convex), $0<p<\infty$, then $\rho_1\cap\rho_2$ is $p$-concave (resp.\ $p$-convex).
\begin{question}\label{qt:amecup}
Let $\XX$ be a quasi-Banach space, and let $\lambda_1$ and $\lambda_2$ be symmetric function quasi-norms such that $(\lambda_1,\XX)$ and $(\lambda_2,\XX)$ are amenable. Is $(\lambda_1\cap\lambda_2,\XX)$ amenable?
\end{question}

Of course, a positive answer to Question~\ref{qt:ame} would yield a positive answer to Question~\ref{qt:amecup}.

\section{The fundamental theorem of calculus}\label{sec:ftc}
\noindent
Let $\XX$ be a quasi-Banach space and let $\lambda$ be a symmetric function quasi-norm such that $(\lambda,\XX)$ is amenable.
If $d\in\NN$, $A\subseteq\RR^d$ is measurable, and $\mu$ is the Lebesgue measure on $A$, we set $L_1^{\lambda}(A,\XX)= L_1^{\lambda}(\mu,\XX)$ and, for $f\in L_1^{\lambda}(A,\XX)$,
$\int_A^\lambda f(x)\, dx = \int_A^\lambda f\, d\mu$. 
A function $f\colon\RR^d \to \XX$ is said to be \emph{locally $\lambda$-integrable} if $f|_A\in L_1^\lambda(A,\XX)$ for every bounded measurable $A\subseteq\RR^d$.

Given $d\in\NN$, we denote by $\Cu$ the set consisting of all $d$-dimensional open cubes. If $y\in\RR^d$, the set $\Cu[y]$ consisting of all $Q\in\Cu$ such that $y\in Q$ is a directed set when ordered by inverse inclusion. We denote by ``$Q\in\Cu\to y$'' the convergence with respect to that directed set.

The following improves \cite{AA2013}*{Theorem 5.2}.
\begin{theorem}\label{thm:ftc}
Let $\XX$ be a quasi-Banach space and $\lambda$ be a symmetric function quasi-norm. Suppose that $\lambda$ is $p$-concave for some $0<p<1$ and that $(\lambda,\XX)$ is amenable. Then, for any locally $\lambda$-integrable function $f\colon\RR^d\to\XX$,
\[
\lim_{Q\in\Cu\to y} \frac{1}{|Q|} \int_Q^\lambda f(x)\, dx=f(y) \quad \text{ a.e.\ } y\in\RR^d.
\]
\end{theorem}

\begin{proof}
Set
\[
M[\XX,\lambda](f)(y)=\sup_{Q\in\Cu[y]} \frac{1}{|Q|} \left\Vert \int_Q^\lambda f(x)\, dx \right\Vert, \quad f\in L_1^\lambda(\RR^d,\XX), \ y\in\RR^d.
\]
If $\kappa$ is the modulus of concavity of $\XX$, we have
\[
M[\XX,\lambda](f+g)\le\kappa M[\XX,\lambda](f)+\kappa M[\XX,\lambda](g),\quad f,g\in L_1^\lambda(\RR^d,\XX).
\]
The result holds for functions in the set
\[
\Ft=\{ x\chi_Q \colon x \in\XX, \ Q\in\Cu\}.
\]
Since $[\Ft]= L_1^\lambda(\RR^d,\XX)$, it suffices to prove that the maximal function $M[\XX,\lambda]$ is bounded from $L_1^\lambda(\RR^d,\XX)$ into $L_{1,\infty}(\RR^d)$.
Let $M$ be the classical Hardy-Littlewood maximal function. Let $f=\sum_{j=1} ^\infty x_j f_j$, where $(x_j)_{j=1}^\infty$ is in the unit ball of $\XX$ and $(f_j)_{j=1}^\infty$ in $L_1(\RR^d)$ satisfies $\lambda( (\Vert f_j\Vert_1)_{j=1}^\infty)<\infty$, be an expansion of $f\in L_1^\lambda(\RR^d,\XX)$. We have
\[
M[\XX,\lambda](f)\le \lambda ((M(f_j))_{j=1}^\infty).
\]
By Theorem~\ref{thm:WL}, the pair $(\lambda,L_{1,\infty}(\RR^d))$ has the MII property. Since $M$ maps $L_1(\RR^d)$ into $L_{1,\infty}(\RR^d)$,
\[
\Vert M[\XX,\lambda](f)\Vert_{1,\infty} \le C_1 \lambda ( (\Vert M(f_j)\Vert_{1,\infty})_{j=1}^\infty )\le C_1 C_2 \lambda ( (\Vert f_j\Vert_1)_{j=1}^\infty ),
\]
where the constants $C_1$ and $C_2$ do not depend on $f$. Consequently, there is constant $C$ such that $\Vert M[\XX,\lambda](f)\Vert_{1,\infty}\le C \Vert f\Vert_{L_1^\lambda(\RR^d,\XX)}$ for all $f\in L_1^\lambda(\RR^d,\XX)$.
\end{proof}

\bibliographystyle{abbrv}
\bibliography{BiblioToward}

\end{document}